\documentclass[12pt]{amsart}  
\usepackage{amsmath}
\usepackage{amsfonts}  
\usepackage{amssymb}  
\usepackage{graphicx} 
\usepackage{MnSymbol}     

\numberwithin{equation}{section}
 
 \newtheorem{thm}{Theorem}[section]
 \newtheorem{lem}[thm]{Lemma}
 \newtheorem{exam}[thm]{Example}
 \newtheorem{prop}[thm]{Proposition}
 \newtheorem{cor}[thm]{Corollary}
 \newtheorem{rem}[thm]{Remark}
 \def\Hol{\mathop{\rm Hol}\nolimits}

 \def\Id{\mathop{\rm Id}\nolimits}

\newcommand{\C}{{\mathbb C}}   
 \newcommand{\R}{{\mathbb R}}           
\newcommand{\N}{{\mathbb N}}        
\newcommand{\Z}{{\mathbb Z}}          
\newcommand{\Q}{{\mathbb Q}}     

\newcommand{\D}{{\mathbb D}}   
\newcommand{\T}{{\mathbb T}}

\author{W. Arendt}
\address{Wolfgang Arendt, Institute of Applied Analysis, University of Ulm. Helmholtzstr. 18, D-89069 Ulm (Germany)} 
\email{wolfgang.arendt@uni-ulm.de}
 
\author{E. Bernard}
\address{Eddy Bernard,  Universit\'e Gustave Eiffel, LAMA, (UMR 8050), UPEM, UPEC, CNRS, F-77454, Marne-la-Vallée (France)}
\email{eddy.bernard@univ-eiffel.fr}

\author{B. C\'elari\`es }
\address{Benjamin C\'elari\`es,  Universit\'e Gustave Eiffel, LAMA, (UMR 8050), UPEM, UPEC, CNRS, F-77454, Marne-la-Vallée (France)}
\email{benjamin.celaries@univ-eiffel.fr}

\author{I. Chalendar}
\address{Isabelle Chalendar,  Universit\'e Gustave Eiffel, LAMA, (UMR 8050), UPEM, UPEC, CNRS, F-77454, Marne-la-Vallée (France)}
\email{isabelle.chalendar@univ-eiffel.fr}

\everymath{\displaystyle}
 
\title[Weighted composition operators induced by rotations]{Spectral properties of weighted composition operators on $\Hol(\D)$ induced by rotations }
\begin{document}	
\begin{abstract} 
	In this article we study the spectrum $\sigma(T)$ and \textit{Waelbroeck spectrum} $\sigma_W(T)$ of a weighted composition operator $T$ induced by a rotation  on $\Hol(\D)$ and given by
	$$Tf(z)=m(z)f(\beta z) \ \ \ (z\in \D)$$
	\\
	where $m\in \Hol(\D)$, $\beta\in \C$, $|\beta | = 1$. If $\beta^n\neq 1$ for all $n\in \N$ we show that $\sigma_W(T)$ is a disc if $m(z_0)=0$ for some $z_0\in \D$ and it is the circle $\{\lambda\in \C :  |\lambda |=|m(0)|\}$ if $m(z)\neq 0$ for all $z\in \D$. We find examples of $m\in A(\D)$ (the disc algebra) such that $\lambda\Id-T$ is invertible in $\Hol(\D)$ (the Fr\'echet space of all holomorphic functions on $\D$), but $(\lambda\Id-T)^{-1}A(\D)\not\subset A(\D)$. Inspired by Bonet \cite{Bonet} we show that $\{\beta^n  :  n\in \N\}\subset \sigma(T)\neq \T$ when the weight is $m\equiv 1$ and $\beta$ a diophantine number. This shows that the spectrum  is not closed in general.
 \end{abstract}	  
\maketitle
    
\section{Introduction}\label{sec:1}
A most popular and successful subject in operator theory is the spectral theory of composition operators. We refer to the monographs \cite{CM95} by Cowen and Mc Cluer and \cite{Sh93} by Shapiro for a comprehensive presentation. In these two books as well as in the vast literature on this subject, these operators are considered on Banach spaces of analytic functions such as the disc algebra $A(\D)$ or the Hardy spaces $H^p(\D)$. More generally, one may multiply by a weight  and consider weighted composition operators as in the classical article \cite{Ka78} by Kamowitz or more recent contributions by Bourdon \cite{Bo12}, Bonet et al. \cite{BGL08}, Chalendar et al. \cite{CGP15}, Galindo et al. \cite{GLW20}, Hyv\"arinen et al. \cite{HLN13}.

They all consider these operators on a Banach space $X$ which is continuously injected in the Fr\'echet space $\Hol(\D)$ of all holomorphic functions on the open unit disc $\D$. The subject of this paper is to study the spectrum directly on the Fr\'echet space $\Hol(\D)$.    In a previous article \cite{ACC20} this had already been done for the pure composition operator $f\mapsto f\circ \varphi$, where $\varphi:\D\to\D$ has a fixed point in the interior and is not an automorphism. Here we consider the case where $\varphi$ is an elliptic automorphism with a  fixed point  in the interior. Due to a similarity transform (see Section~\ref{sec:8} below) we may take a rotation as $\varphi$. Thus our setting is the following. We take $\beta\in\C$ such that $|\beta|=1$ and $m\in\Hol(\D)$ to study spectral properties of the operator 
\[  T:\Hol(\D)  \to \Hol(\D) \mbox{ given by }Tf(z)=m(z)f(\beta z),\,\,\,\, z\in\D.   \]
It turns out that this leads to quite a variety of results depending on the case where $\beta$ is \emph{periodic} (i.e. $\beta^N=1$ for some integer $N\geq 2$) or \emph{aperiodic} (i.e. $\beta^n\neq 1$ for all $n\in\N$), and on properties of $m$. At first we study the eigenvalues (Section~\ref{sec:3}).  We can describe those completely in the periodic as well as the aperiodic case and also determine the eigenspaces.

Following an idea of Bonet \cite{Bonet}  we then consider an aperiodic $\beta$ of the form $\beta=e^{2i\pi\xi}$ where $\xi$ is a diophantine number. For the weight $m\equiv 1$ we show that 
\[  \{  \beta^n:n\in\N_0\}   \subset \sigma(T)\subset \T\setminus \{  e^{2i\pi r}:r\in\Q, e^{2i\pi r}\neq 1 \}, \]
where $\sigma(T):=\{   \lambda\in\C:\lambda\Id -T\mbox{ is not bijective}\}$ denotes the spectrum of $T$ and where $\N_0=\N\cup\{0\}$. Thus $\overline{\sigma(T)}=\T:=\{\lambda\in\C:|\lambda|=1\}$ and $\overline{\T\setminus \sigma(T)}=\T$. This case already demonstrates how bad spectral behaviour in Fr\'echet space can be: the spectrum may not be closed and the resolvent may not be continuous. For this reason, as in the specialized literature (see e.g. the monograph \cite{Va82} by Vasilescu), we consider the \textit{Waelbroeck spectrum} $\sigma_W(T)$ and its complement, the \textit{Waelbroeck resolvent set} $\rho_W(T)$ of $T$. It can be defined in arbitrary Fr\'echet spaces but we prefer to work only in $\Hol(\D)$ using advantages of this special space. Here we may define 
\begin{eqnarray*}
	   \rho_W(T) &  = &  \{  \lambda\in\rho(T):\exists \delta>0\mbox{ such that }D(\lambda,\delta)\subset \rho(T) \\
 &  & 	   \mbox { and }\sup_{\mu\in D(\lambda,\delta)}\sup_{|z|\leq r}| ((\mu \Id-T)^{-1}f)(z)|<\infty \\
  &  & \mbox{ for all }f\in\Hol (\D),r<1  \}
\end{eqnarray*}	             

This set is open, and so its complement $\sigma_W(T)$ is closed. The most interesting is the aperiodic case. There we consider two cases.   If $m(z)\neq 0$ for all $z\in\D$, then 
\[\sigma_W(T)=\{\lambda\in\C :|\lambda|=|m(0)|\}.\]  
Much more difficult is the case where $m$ has a zero in $\D$. Our main result, Theorem~\ref{th:7.5}, shows that then $\sigma_W(T)$ is a disc. In fact, let
\[    M_r =\exp \left(   \frac{1}{2\pi}\int_0^{2\pi} \log |m(re^{it})|dt   \right)   \]
for $0<r<1$, and $M_1:=\sup_{0<r<1}M_r\in[0,\infty]$. Then 
\[    \sigma_W (T)=\{    \lambda \in\C:|\lambda|\leq M_1\}\]
if $M_1<\infty$ and $\sigma_W(T)=\C$ otherwise. 

Before proving this result, we turn to Banach spaces first. In Section~\ref{sec:5} we consider the restriction $T_X$ of $T$ to $X=A(\D)$ (requiring $m\in A(\D)$) and also $X=H^p(\D)$, for example. 
We show that in the aperiodic case, if $m\in A(\D)$ has a zero in $\overline{\D}$, then 
\[  \sigma(T_{A(\D)}) =\{  \lambda\in\C : |\lambda|\leq M^*\} ,\]
where $M^*=\frac{1}{2\pi}\int_0^{2\pi} \log |m(e^{it})|dt .$ This result is given by Kamowitz \cite[Theorem 4.8 and 4.9]{Ka78} with quite complicated, and actually delicate arguments. Here we give a very simple proof 
based on spectral decomposition. And it is also this spectral decomposition which leads to a proof in the case of $\Hol(\D)$ in Section~\ref{sec:7}. 

We devote a special section, Section~\ref{sec:6}, to establish the spectral decomposition theorem in $\Hol(\D)$ and also to describe the \textit{Waelbroeck spectrum} for arbitrary linear operators on the space $\Hol(\D)$. 

It is worth it to compare the results on $\Hol(\D)$ and the spaces $X=A(\D),H^p(\D)$. In fact, it can happen that for $m\in A(\D)$, $M_1<M^*$. Then for $M_1<|\lambda|\leq M^*$, $\lambda\in \rho_W(T)$ but $\lambda\in\sigma(T_X)$ for $X=A(\D)$ or $H^p(\D)$, $1\leq p\leq \infty$.  This means that $(\lambda \Id - T)^{-1} X\not\subset X $ for any such $\lambda$ and any such space.            
\section{Rotational invariance of the spectrum}\label{sec:2}
Let $\beta\in\C$, $|\beta|=1$ and let $m\in\Hol(\D)$, $m\neq 0$. We consider the operator 
\[   T:\Hol(\D)\to \Hol(\D), \mbox{ given by  }  
 (Tf)(z)=m(z)f(\beta z).\]
By $\rho(T):=\{\lambda\in\C: (\lambda\Id -T)\mbox{ is bijective}\}$ we denote the \emph{resolvent set} of $T$. 
We let \[R_\lambda:=(\lambda\Id -T)^{-1}\in {\mathcal L}(\Hol(\D)), \;\;\;\;\lambda\in \rho(T).\] 
The continuity of $R_\lambda$ is a consequence of the closed graph theorem. Denote by 
\[ \sigma(T):=\C\setminus \rho(T)\] the \emph{spectrum} of $T$.  

The spectrum is invariant by rotation by $\beta$. More precisely we have the following.
\begin{thm}\label{th:2.1}
	\[\sigma(T)=\beta \sigma(T)\cup \{m(0)\}.\]
\end{thm} 
Straightforward  properties of invariance of the spectrum are the following.

\begin{cor}\label{cor:2.2}
 \,\\ \vspace{-0,5cm}
\begin{itemize}
\item[a)] If $\lambda\in \sigma(T)$ then $\lambda^n\in\sigma(T)$ for all $n\in\N$.  
\item[b)] $\{ \beta^n m(0):n\in\N_0  \}\subset \sigma(T)$. 
\end{itemize}
\end{cor}   
For the proof of Theorem~\ref{th:2.1} we use the restriction $T_0$ of $T$ to the space 
\[  \Hol_0(\D)  :=\{   f\in\Hol(\D):f(0)=0\}.\]
Since $(Tf)(0)=m(0)f(0)$, it follows that $T\Hol_0(\D)\subset \Hol_0(\D)$, and then $T_0\in{\mathcal L}(\Hol_0(\D))$. The next lemma 
shows the link between the spectrum of $T_0$ and the one of $T$.  
\begin{lem}\label{lem:2.3}
\[  \sigma(T_0)=\beta\sigma (T)\mbox{ and }\rho(T_0)=\beta\rho(T).  \]
\end{lem} 
\begin{proof}
We consider the bijective and linear map $\Phi:\Hol(\D)\to\Hol_0(\D)$ given by $(\Phi f)(z)=zf(z)$ whose inverse is defined by 
$(\Phi^{-1}g)(z)=\frac{1}{z}g(z),\,\, z\neq 0$.  Then $\Phi^{-1} T_0 \Phi=\beta T$. Since $T_0$ and $\beta T$ are similar, the assertions of the lemma follow.  	
	
\end{proof}	
A further step in the proof of Theorem~\ref{th:2.1} is the following. 
\begin{lem}\label{lem:2.4}
 \,\\ \vspace{-0,5cm}
\begin{itemize}
	\item[a)]  $m(0)\in\sigma(T)$,  
	\item[b)]  $\rho(T)\subset \rho(T_0)$ and 
	\item[c)] $\rho(T_0)\subset \rho(T)\cup \{  m(0)\}$. 
\end{itemize}
\end{lem}
\begin{proof}
a) Since $((m(0)\Id -T)f)(0)=0$ for all $f\in\Hol(\D)$, the operator $m(0)\Id -T$ is not surjective. \\
b) Let $\lambda\in\rho(T)$. Then $m(0)\neq \lambda$ by a).  Let $g\in\Hol_0(\D)$, $f= R_\lambda g$. Then 
\[\lambda f(z)-m(z)f(\beta z)=g(z)\mbox{ for all }z\in\D .\]
Hence $(\lambda -m(0))f(0)=g(0)=0$. Thus $f(0)=0$, which proves that $R_\lambda \Hol_0(\D)\subset \Hol_0(\D)$ and proves the desired inclusion. \\
c) Let $m(0)\neq \lambda\in\rho(T_0)$. Then $\lambda\Id - T$ is injective. In fact, if $(\lambda\Id-T)f=0$, since $\lambda \neq m(0)$, evaluating in $0$ show that $f\in \Hol_0(\D)$. Hence, $f=0$ since $\lambda \in \rho(T_0)$. In order to show surjectivity, since $\Hol(\D) =\Hol_0(\D)\oplus \C$, it suffices to show  that  there exists $f\in \Hol(\D)$ such that $(\lambda\Id -T)f={e_0}$, 
 where $e_0$ is the function equal to 1 on $\D$.  
Evaluating in $0$ shows that $(\lambda -m(0))f(0)=1$, and then $f(0)=\frac{1}{\lambda -m(0)}$.  Moreover
\begin{equation}\label{eq:2.1}
(\lambda \Id -T)f(0)e_0 +(\lambda\Id -T)(f-f(0))e_0 =e_0.
\end{equation}     
Since $(\lambda\Id -T)f(0)e_0 + (\lambda\Id -T)(f-f(0)e_0)-f(0)(m -m(0)e_0)$, (\ref{eq:2.1}) is equivalent to 
\[  (\lambda \Id -T)(f-f(0)e_0)= f(0) (m-m(0)e_0) .\]
Since $\lambda\in \rho(T_0)$, there exists a unique $f_1\in\Hol_0(\D)$ such that 
\[(\lambda\Id -T)f_1=\frac{m-m(0)e_0}{\lambda -m(0)}.\]   
Thus $f:= f_1 + \frac{1}{\lambda -m(0)}e_0$  is the unique function in $\Hol(\D)$ satisfying \\
$\lambda f-Tf=e_0$.   
\end{proof}
If $\beta=e^{2i\pi\alpha}$, $m(z)=1$ for all $z\in\D$ and $\alpha$ is a diophantine number, it was shown by J. Bonet 
\cite[Corollary 3]{Bonet} that 
$m(0)=1\in \rho(T_0)$. So Lemma~\ref{lem:2.4} is optimal.  \\

Theorem~\ref{th:2.1} is an immediate consequence of Lemma~\ref{lem:2.3} and \ref{lem:2.4}. 

 We will see in Section~\ref{sec:4} that $\sigma(T)$ is not closed in general. 
 We add a spectral result, which is more precise than Corollary~\ref{cor:2.2} and will be useful later. 
 
 By $e_k$ we denote the function $e_k(z)=z^k$ $(z\in\D)$, where $k\in\N_0$.   

\begin{lem}\label{lem:2.5}
Let $m\in \Hol(\D)$ such that $m(0)\neq 0$, $\beta\in\T$, $n\in\N$ such that $\beta^k\neq 1$ for $k=1,\cdots, n$. Then 
\[ e_k\not\in (\beta^k m(0)e_0-T)\Hol(\D)\mbox{ for }k=1,\cdots,n.  \] 	
\end{lem}
\begin{proof}
 First note that $e_0\not\in (m(0)e_0-T)\Hol(\D)$	since for $f\in\Hol(\D)$, $((m(0)e_0-T)f)(0)=0$. 
Now  assume that there exists $k\in \{1,\cdots,n\} $ such that  \[e_k\in (\beta^k m(0)e_0-T)\Hol(\D).\]
 Choose $k$ minimal with this property. So there exists $f\in\Hol(\D)$ such that 
\[\beta^k m(0)f(z)-m(z)f(\beta z)=z^k,\,\,\, z\in\D.\]
Then $\beta^k m(0)f(0)-m(0)f(0)=0$. Thus $f(0)=0$ and and there exists $f_1\in\Hol(\D)$ such that $f=e_1f_1$. It follows that 
\[ \beta^k m(0)zf_1(z) -m(z)\beta zf_1(\beta z)=z^k,\,\,\, z\in\D .  \]  
Then \[ \beta^{k-1} m(0)(\beta f_1)(z) -m(z)(\beta f_1)(\beta z)=z^{k-1},\,\,\, z\in\D ,  \]
which contradicts the minimality of $k$. 
\end{proof}
\section{The point spectrum of weighted composition operators induced by periodic and aperiodic  rotations}\label{sec:3}
The main purpose of this section is to determine the point spectrum. At first we consider the periodic case, where it is easy to determine 
the entire spectrum, which we do first. The point spectrum demands more efforts. We let $\beta\in \T$, $m\in\Hol(\D)$, $m\neq 0$, and consider 
the operator $T$ on $\Hol(\D)$ by 
\[   (Tf)(z)= m(z)f(\beta z),\,\,\, z\in\D.\]
       We first show the following, which holds for the periodic and aperiodic case. By 
       \[\sigma_p(T):=\{ \lambda\in\C :\lambda \Id -T\mbox{ is not injective} \} \]
       	we denote the \emph{point spectrum} of $T$.  As in the previous section, by $e_k$ we denote the function $e_k(z)=z^k$ $(z\in\D)$, where $k\in\N_0$.   
 \begin{prop}\label{prop:3.1}
 \,\\ \vspace{-0,5cm}
 \begin{itemize}
 	\item[a)]   $0\not\in\sigma_p(T)$.  
 	\item[b)]   If $\lambda\in\sigma_p(T)$, then $\beta^n\lambda\in\sigma_p(T)$ for all $n\in\N$. 
 \end{itemize}	
 \end{prop}      	
\begin{proof}
a) Let $f\in\Hol(\D)$ such that $(Tf)(z)=m(z)f(\beta z)=0$ for all $z\in\D$. Then $f(w)=0$ for all $w\in \beta\{ z:m(z)\neq 0\}=:\Omega$. Since $m\neq 0$, 
$\Omega$ is open and non-empty. It follows that $f=0$.\\
b) Let $0\neq f\in \Hol(\D)$ such that $Tf=\lambda f$. Then $f_k:=e_kf\neq 0$ and $Tf_k=\beta^k \lambda f_k$.  
\end{proof}
\subsection{The periodic case}
The spectrum $\sigma(M)$ of the multiplication operator $M\in {\mathcal L}(\Hol(\D))$ given by $Mf=mf$ is clearly  equal to $m(\D)$. Thus the spectrum of $M$ is open unless 
$m$ is constant. Now, let $\beta\in\C$ be a root of unity different from $1$.  In other words there exists $N\geq 2$ such that  $\beta^N=1$, $\beta^k\neq 1$ for $k=1,\cdots,N-1$. 
Then $T^N$ is the multiplication operator given by 
\[T^Nf=m_N f \mbox{ where }m_N(z)=m(z)m(\beta z)\cdots m(\beta^{N-1}z).\]
\begin{prop}\label{prop:3.2}
	One has $\sigma(T)=\{    \lambda \in\C :\lambda^N\in m_N(\D)\}$. 
\end{prop}  
\begin{proof}
Since $\sigma(T)^N=\sigma(T^N)=m_N(\D)$, one has $\lambda^N\in m_N(\D)$ for all $\lambda\in \sigma(T)$. For the converse, note that, by Theorem~\ref{th:2.1}, $\beta \sigma(T)\subset \sigma(T)$. 
Since $\beta$ is an $N$-th root of unity, this implies that $\beta \sigma(T)=\sigma(T)$. Now let $\lambda\in\C$ such that $\lambda^N\in m_N(\D)$. Then $\lambda^N\in\sigma(T^N)$.
So there exists $\mu\in \sigma(T)$ such that $\lambda^N=\mu^N$. Hence $\lambda=\beta^k\mu$ for some $k\in\N$ and so $\lambda\in\sigma(T)$.  	
\end{proof}	
\begin{prop}\label{prop:3.3}
	Assume that $m(z_0)=0$ for some $z_0\in\D$. Then $\sigma_p(T)=\emptyset$. 
\end{prop}
\begin{proof}
We know from Proposition~\ref{prop:3.1}	that $0\not\in\sigma_p(T)$. Assume that $0\neq \lambda\in\sigma_p(T)$. Then $\lambda^N\in \sigma_p(T^N)$. There exists $0\neq f\in\Hol(\D)$ such that 
\[ m_N(z)f(z)=\lambda^N f(z)\mbox{ for all }z\in\D.  \]
This implies that $m_N(z)=\lambda^N$ if $f(z)\neq 0$. Since $f\neq 0$, $f$ has merely isolated zeros. Thus $m_N(z)=\lambda^N$ for all $z\in\D$. Consequently \[\lambda^N=m_N(z_0)=0,\] a contradiction.      
\end{proof}	
Now we can describe the point spectrum and the eigenspaces of $T$ in the periodic case. 
\begin{thm}\label{th:3.4}
	Let $\beta\in\C$, $N\in\N$, $N\geq 2$ such that $\beta^N=1$, $\beta^k\neq 1$ for $k=1,\cdots, N-1$. Let $m\in\Hol(\D)$, $m\neq 0$ and consider the operator 
	\[    T:\Hol(\D)\to \Hol(\D)\mbox{ given by }  (Tf)(z)=m(z)f(\beta z)\,\,\, (z\in\D).\]
	The following are equivalent.
\begin{itemize}
	\item[a)] $\sigma_p(T)\neq \emptyset$;
	\item[b)] there exist $a_0\in\C$, $f_j\in\Hol(\D)$, $j=1,\cdots,N-1$ such that $m=\exp m_1$ where 
	\[m_1(z)=a_0 +zf_1(z^N) +z^2f_2(z^N)+\cdots + z^{N-1}f_{N-1}(z^N)\,\,\, (z\in\D).\]
\end{itemize}	
In that case \[T^N =e^{Na_0}\Id,\,\, \sigma_p(T)=\{e^{a_0}\beta^k : k=0,\cdots , N-1\}=\sigma(T)\]
and 
\begin{eqnarray*}
\ker (T-e^{a_0} \beta^k \Id) & = & \{ f(z)=z^k\exp{g_1(z)}: \widehat{g_1}(n)=\frac{\widehat{m_1}(n)}{1-\beta^n} \mbox{ if }n\not\in N\N_0 \mbox{ and }\\
  &  & \limsup_{n\to\infty}  |\widehat{g_1}(Nn)|^{1/n}\leq 1 \},
\end{eqnarray*}
showing that the dimension of the eigenspaces are infinite. 
\end{thm} 
For the proof we need to characterize when the function $m_N$ is constant. 
\begin{lem}\label{lem:3.5}
Let $\beta\in\C$ such that $\beta^N=1$, $\beta^k\neq 1$ for $k=1,\cdots, N-1$ where $N\geq 2$. Let $m\in\Hol(\D)$. The following assertions are equivalent.
\begin{itemize}
	\item[(i)] $\exists c\in\C\setminus\{0\}$ such that $m_N(z)=c$ for all $z\in\D$;
	\item[(ii)] there exist $a_0\in\C$, $f_1,\cdots, f_{N-1}\in\Hol(\D)$ such that 
	\[ m=\exp(m_1)\mbox{ and }m_1(z)=a_0+zf_1(z^N)+\cdots + z^{N-1}f_{N-1}(z^N). \]
\end{itemize}
\end{lem} 
\begin{proof}
$(i)\Rightarrow (ii)$ By $(i)$ $m(z)\neq 0$ for all $z\in\D$. Thus there exists $m_1\in\Hol(\D)$ such that $m=\exp(m_1)$ and 
\[   c=\exp (m_1(z) + m_1(\beta z) +\cdots +m_1(\beta^{N-1}z))\,\,\, \mbox{ for all }z\in\D. \]
Thus 
\begin{equation}\label{eq:const}
m_1(z)+m_1(\beta z)+\cdots + m_1(\beta^{N-1}z)=e^{c_1} +2i\pi k(z),
\end{equation} 	
where $c=e^{c_1}$ and $k:\D\to \Z$ is continuous, hence constant.  Let $m_1(z)=\sum_{n\geq 0} a_n z^n$. Then 
\[ c_2=\sum_{n\geq 0}  a_n z^n (1+\beta^n +\beta^{2n}+\cdots +\beta^{(N-1)n})\,\,\, \mbox{for all }z\in\D\mbox{ and some }c_2\in\C.  \]
Thus if $\beta^n=1$, i.e. $n\in N\N$, then $a_n=0$, whereas for $\beta^n\neq 1$, 
 since \[1+\beta^n +\beta^{2n}+\cdots +\beta^{(N-1)n}=\frac{1-\beta^{nN}}{1-\beta^n}=0,\]
  $m_1$ satisfies (\ref{eq:const}) if and only if $m_1$ is of the form 
 \[  m_1(z)=\sum_{n\geq 0,n\not\in N\N} a_n z^n.\] 
 Let $f_k(z)=\sum_{m\geq 0} a_{mN+k}z^{mN}$ for $k=1,\cdots, N-1$. Then $(ii)$ is fulfilled.\\
 $(ii)\Rightarrow (i)$ We have to show that 
 \[  M(z):=m_1(z)+m_1(\beta z)+\cdots +m(\beta^{N-1}z)\mbox{ is constant for all }z\in\D .  \]
 We have $M(0)=Na_0$. Since $1+\beta +\cdots +\beta^{N-1}=0$, one has 
 \[    zf_1(z^N) +\beta zf_1(\beta^N z^N)+\cdots +\beta^{N-1}zf_1(\beta^N z^N) =(1+\beta+\cdots +\beta^{N-1} )zf_1(z^N)=0.  \]
 	Similarly $1+\beta^k +\cdots + (\beta^k)^{N-1}=0$ and so 
\[ z^kf_k(z^N) +(\beta z)^k f_k((\beta z)^N) +\cdots + (\beta^{N-1}z)^kf_k((\beta z)^N)=0,    \]
for $k=1,\cdots,N-1$. This proves the claim.
\end{proof}	
\begin{proof}[Proof of Theorem~\ref{th:3.4}]
$a)\Rightarrow b)$ Assume that $\sigma_p(T)\neq \emptyset$. Then by Proposition~\ref{prop:3.3}, $m(z)\neq 0$ for all $z\in\D$ and $\sigma_p(T^N)\neq \emptyset$. 
Thus there exist $0\neq f\in\Hol(\D)$, $0\neq \lambda\in\C$ such that $m_N(z)f(z)=\lambda f(z)$ $(z\in\D)$. This implies that $m_N(z)=\lambda$ for all $z\in\D$, and $b)$ follows 
from Lemma~\ref{lem:3.5}.\\
$b)\Rightarrow a)$ It follows from $b)$ that $T^N=e^{Na_0}\Id$. Thus $\sigma_p(T^N)=\sigma(T^N)=e^{Na_0}$, which implies that 
\[  \sigma_p(T)   = \{e^{a_0}  \beta^k :k=0,\cdots, N-1 \}=\sigma(T),\]
since by Proposition~\ref{prop:3.1}, $\beta\sigma_p(T)\subset\sigma_p(T)$ and hence $\beta \sigma_p(T)=\sigma_p(T)$ (using $\beta^N =1$). Similarly $\beta\sigma(T)=\sigma(T)$ by Theorem~\ref{th:2.1}.

Now  let us describe the functions $f\in\Hol(\D)$, $f\neq 0$, solutions of the equation 
\begin{equation}\label{eq:eigen1}
\exp (m_1(z))f(\beta z)=e^{a_0}\beta^k f(z) \mbox{ with }m_1(0)=a_0, k\in\{0,\cdots,N-1\}.
\end{equation}
Since $f\neq 0$, there exists $n_0\in\N_0$ such that $f(z)=z^{n_0} g(z)$  on $\D$, where $g\in \Hol(\D)$ and $g(0)\neq 0$. Let $0<r<1$ such that $g(z)\neq 0$ for all $z\in D(0,r)$. 
Then there exists $g_1\in\Hol(D(0,r))$ such that  $g(z)=e^{g_1(z)}$ on $D(0,r)$. Then (\ref{eq:eigen1}) becomes 
\begin{equation}\label{eq:eigen2}
\exp (m_1(z))z^{n_0}\beta^{n_0}\exp(g_1(\beta z))=e^{a_0}\beta^k z^{n_0}\exp (g_1(z)) \mbox{ if }|z|<r,  
\end{equation} 
that is 
\begin{equation}\label{eq:eigen3}
\exp (m_1(z) + g_1(\beta z))=e^{a_0}\beta^{k-n_0}  \exp (g_1(z)) \mbox{ if }|z|<r.  
\end{equation} 
Taking $z=0$ it follows that  $n_0=k$ and then $(\ref{eq:eigen3})$ becomes 
\begin{equation}\label{eq:eigen4}
\exp (m_1(z) + g_1(\beta z))=e^{a_0} \exp (g_1(z)) \mbox{ if }|z|<r.  
\end{equation} 
Since $\widehat{m_1}(n)=0$ for all $n\in N\N$ and $\beta^N=1$, the only condition on $\widehat{g_1}(n)$ with $n\in N\N$ is given by the Hadamard formula on the radius of convergence, 
namely \[\limsup_{n\to\infty}|\widehat{g_1}(nN)|^{1/(nN)}\leq 1,\]
 which is equivalent to   $\limsup_{n\to\infty}|\widehat{g_1}(nN)|^{1/n}\leq 1$. Moreover, (\ref{eq:eigen4}) implies that 
 \[\widehat{g_1}(n)=\frac{\widehat{m_1}(n)}{1-\beta^n} \mbox{ if }n\not\in N\N_0.\]
 Since $m_1\in\Hol(\D)$ and since there exists $\delta>0$ such that  $\delta\leq |1-\beta^n|\leq 2$ for all $n\not\in N\N_0$, a  function $g_1$ such that 
 \[  \limsup_{n\to\infty}|\widehat{g_1}(nN)|^{1/n}\leq 1\mbox{ and }\widehat{g_1}(n)=\frac{\widehat{m_1}(n)}{1-\beta^n} \mbox{ if }n\not\in N\N_0 \] is in $\Hol(\D)$. This proves the last assertion
  of the theorem.  
 \end{proof}	
\subsection{The aperiodic case}
Here we assume that $\beta\in\C$, $|\beta|=1$ and $\beta^n\neq 1$ for all $n\in\N$. Our aim is to determine the point spectrum $\sigma_p(T)$ of $T$. We first show that also in the aperiodic 
case, $\sigma_p(T)=\emptyset$ whenever $m$ has a zero. 
\begin{prop}\label{prop:3.6}
If there exists $z_0\in\D$ such that $m(z_0)=0$, then $\sigma_p(T)=\emptyset$.  
\end{prop}
\begin{proof}
By Proposition~\ref{prop:3.1} a), $0\not\in\sigma_p(T)$. Let $0\neq \lambda\in\C$. 	Assume that $m(z_0)=0$ and let $f\in\Hol(\D)$ such that 
\[  m(z)f(\beta z)=\lambda f(z)\mbox{ for all }z\in\D. \]
Then $f(z_0)=0$ and thus $0=m(\overline{\beta}z_0)f(z_0)=\lambda f(\overline{\beta}z_0)$. Iterating the argument we see that $f(\overline{\beta}^kz_0)=0$ for all $k\in\N_0$. Thus $f$ vanishes on 
$z_0\T$ which is the closure of $\{\overline{\beta}^k z_0\ : \ k \in\N_0 \}$. This implies that $f=0$.  
\end{proof}
Now assume that $m(z)\neq 0$ for all $z\in\D$. Then there exists $m_1\in\Hol(\D)$ such that 
\[   m(z)=\exp(m_1(z))   \mbox{ for all }z\in\D. \]
Let $m_1(z)=\sum_{n\geq 0}a_n z^n$ for all $z\in\D$.  Then the following holds.
\begin{thm}\label{th:3.7}
 \begin{itemize}
	\item[a)] If \begin{equation}\label{eq:3.1}
	                        \limsup_{n\to\infty}\left(     \frac{|a_n|}{|1-\beta^n|}\right)^{1/n}\leq 1
               	\end{equation} 
      then $\sigma_p(T)=\{ m(0)\beta^n :n\in\N_0\}$ and $\ker (T-m(0)\beta^n\Id)$ is one-dimensional, generated by the function
      \begin{equation}\label{eq:eigen42}
          f_n(z) = z^n \exp \left( g_1(z) \right) \text{ where } g_1(z) = \sum_{n \geqslant 0} \frac{a_n}{1-\beta^n} z^n ~ \text{ for all } z \in \D .
      \end{equation}

    \item[b)]   if    \begin{equation}\label{eq:3.2}
    \limsup_{n\to\infty}\left(     \frac{|a_n|}{|1-\beta^n|}\right)^{1/n}>1 
    \end{equation}   	
    then $\sigma_p(T)=\emptyset$. 
\end{itemize}
\end{thm}  
\begin{proof}
Let $\lambda\in\sigma_p(T)$. Then $\lambda\neq 0$ by Proposition~\ref{prop:3.1} and there exists $0\neq f\in\Hol(\D)$ such that 
\begin{equation}\label{eq:3;3}
m(z)f(\beta z)=\lambda f(z)\,\,\, (z\in\D).
\end{equation}  	
This implies that $f(z)\neq 0$ for all $z\in\D\setminus \{0\}$. In fact, assume that $f(z_0)=0$ where $z_0\in\D\setminus\{0\}$. Then by (\ref{eq:eigen2}), 
$f(\beta z_0)=0$. Iterating this gives $f(\beta^k z_0)=0$ for all $k\in\N$. Since the closure of $\{ \beta^k z_0 :k\in\N_0 \}$ is a circle, $f=0$, a contradiction.  Consequently there exist $M\in\N_0$ and  $f_1\in\Hol(\D)$ such that 
\[  f(z)=z^M \exp( f_1(z)  )\,\,\, (z\in\D). \]
Choose  $\gamma\in\C$ such that $\lambda=e^{\gamma}$. Let $f_1(z)=\sum_{n\geq 0} b_n z^n$. Then for $\beta=e^{i\alpha} $,
\[e^{m_1(z)}~ z^M ~ e^{i \alpha M} \exp \left( f_1(\beta z) \right) = e^\gamma ~z^M \exp(f_1(z))   .\]
So, there exists $k(z) \in \Z$ such that
\begin{equation}\label{eq:3.9}
    m_1(z) + i \alpha M + f_1(\beta z) = \gamma + f_1(z) + 2i\pi k(z).
\end{equation}
By continuity, $k(z) = k(0) :=k_0$ for all $z \in \D$. Letting $z=0$ in (\ref{eq:3.9}) yields 
\[ a_0 + i \alpha M - \gamma = 2i \pi k_0 .\]
Hence, $\exp \left(a_0 + i \alpha M - \gamma \right) = 1 $. Recall that $\exp(a_0) = m(0)$. Thus, $m(0)e^{i \alpha M} = e^\gamma = \lambda$. Now, (\ref{eq:3.9}) implies that
\[ \sum_{n \geqslant 1} a_n z^n + \sum_{n \geqslant 1} b_n \beta^n z^n - \sum_{n \geqslant 1} b_n z^n = 0.\]
Hence, $b_n = \frac{a_n}{1-\beta^n}$ for $n \in \N$. Thus, $\sum_{n \geqslant 1} \frac{a_n}{1-\beta^n} z^n$ has a radius of convergence $\geqslant 1$, i.e. 
$\limsup_{n\to\infty}\left(     \frac{|a_n|}{|1-\beta^n|}\right)^{1/n} \leqslant 1$. Thus, (\ref{eq:eigen42}) holds and $\lambda = m(0) e^{i \alpha M}$.

Conversely, assume (\ref{eq:3.1}). Let $M \in \N_0$ and let 
\[ f(z) := z^M \exp\left\{ \sum_{n=1}^{+ \infty} \frac{a_n}{1-\beta^n} z^n \right\} . \]
Then in virtue of (\ref{eq:3.1}), $f \in \Hol(\D)$, $f \neq 0$ and
\begin{eqnarray*}
(TF)(z) & = & m(z) f(\beta z) \\
& = & \exp \left( \sum_{n=0}^{+ \infty} a_n z^n \right) \beta^M z^M \exp \left( \sum_{n=0}^{+ \infty} \frac{a_n}{1-\beta^n} z^n \right)\\
& = & m(0) \beta^M z^M \exp \left( \sum_{n=0}^{+ \infty} \left( 1+ \frac{a_n}{1-\beta^n} \right) z^n \right)\\
& = & m(0) \beta^M f(z).
\end{eqnarray*}
This shows that $m(0) \beta^M \in \sigma_p(T)$ for all $M \in \N_0$ if (\ref{eq:3.1}) holds.
\end{proof}	

\section{The spectrum of composition operators induced by aperiodic rotation on $\Hol(\D)$: diophantine arguments}\label{sec:4}
Let $\beta\in\C$, $|\beta|=1$, and assume that $\beta$ is not a root of unity, i.e. $\beta^n\neq 1$ for all $n\in\N$. 
Let $\varphi:\D\to\D$ be given by  $\varphi(z)=\beta z$. Such $\varphi$ is called an \emph{aperiodic rotation}. Then define 
\[  C_\varphi:\Hol(\D)\to \Hol(\D) \mbox{ given by }C_\varphi f=f\circ \varphi .\]
We first determine the point spectrum. 
\begin{prop}\label{prop:4.1}
$\sigma_p(C_\varphi)=\{\beta^n: n\in\N_0\}$	
\end{prop}
\begin{proof}
This a corollary of Theorem~\ref{th:3.7} or \cite[Proposition~1]{Bonet} 	
with an easy direct proof. 
\end{proof}	

We start by the following result from \cite{Bonet} and add a proof to keep the article self contained.

\begin{prop}\label{prop:4.1prime}
$\sigma (C_\varphi) \subset \T$.
\end{prop}

\begin{proof}
    We shall prove that $\C - \T \subset \rho(C_\varphi)$. Let $\lambda \in \C$ such that $\vert \lambda \vert \neq 1$. Let $g \in \Hol(\D)$ and write $g(z) = \sum_{n \geqslant 0} b_n z^n$ with $\limsup_{n \rightarrow + \infty} \vert b_n \vert^{\frac{1}{n}} \leqslant 1$. We aim to solve
    \begin{equation}\label{eq:4.1'.1}
        C_\varphi f - \lambda f = g
    \end{equation}
    with $f \in \Hol(\D)$ given by $f(z) = \sum_{n \geqslant 0} a_n z^n$. From (\ref{eq:4.1'.1}), we get
    \[ \forall n \in \N_0, ~ a_n \beta^n - \lambda a_n = b_n \Leftrightarrow a_n = \frac{b_n}{\beta^n - \lambda} .\]
    However, for each $n \in \N_0$, $ \left\vert \beta^n - \lambda \right\vert \geqslant \left\vert 1 - \vert \lambda \vert \right\vert $. Thus, $\limsup_{n \rightarrow + \infty} \vert a_n \vert^{\frac{1}{n}} = \limsup_{n \rightarrow + \infty} \vert b_n \vert^{\frac{1}{n}} \leqslant 1$. We have proved that "$\forall g \in \Hol(\D), \exists ! f \in \Hol(\D)$ such that $C_\varphi f - \lambda f = g$". Hence, $\lambda \in \rho(T)$.
\end{proof}

Inspired by Bonet \cite{Bonet} we want to study the spectrum of $C_\varphi$ if $\beta=e^{2i\pi\xi}$ where $\xi$ is a diophantine number. Let $\tau>2$ and 
\[ {\mathcal D}(\tau):=\{   \xi\in\R:\exists   \gamma>0\mbox{ such that }|p/q -\xi|\geq \gamma q^{-\tau}\mbox{ for all }p\in\Z,q\in\N\}\] be 
the set of all \emph{diophantine numbers} of order $\tau$. It is well-known that $\R\setminus  \cup_{\tau>2}{\mathcal D}(\tau)$ has Lebesgue measure $0$. 
\begin{thm}\label{th:4.2}
	Let $\tau>2$, $\xi\in {\mathcal D}(\tau)$. Define $\varphi(z)=e^{2i\pi\xi} z $ and consider the operator $C_\varphi$ on $\Hol(\D)$ given
	 by $C_\varphi f=f\circ \varphi$. Then, 
	 \[ \left\{ \beta^n : n \in \N_0 \right\} \subset \sigma(C_\varphi) \subset \left\{ e^{2i\pi x} : x \notin \Q \right\} \cup \left\{ 1 \right\} . \]
\end{thm} 
\begin{proof}Thanks to Propositions \ref{prop:4.1} and \ref{prop:4.1prime}, it suffices to show that
	 \[   \{\lambda\in\C\setminus\{1\}:\exists r\in\Q, \lambda=e^{2i\pi r}\}\subset \rho(C_\varphi) .  \]
	 Let $\lambda = e^{2 i \pi r} \neq 1$ with $r \in \Q$.

Let $\gamma>0$ such that $|\xi-p/q|\geq \gamma q^{-\tau}$ for all $p\in\Z$ and $q\in\N$. Let $r=\frac{p_0}{q_0}$, $p_0\in\Z$, $q_0\in\N$. We claim that 
\[  \frac{1}{|\lambda -\beta^k|}\leq \left( \frac{q_0^{\tau}}{4\gamma}      \right) k^{\tau-1}\,\,\, (k\in\N),    \]
where $\beta:=e^{2i\pi\xi}$. Indeed, since $\sin$ is concave on $[0,\pi/2]$, one has $|\sin x|\geq \frac{2}{\pi}|x|$ for $|x|\leq \frac{\pi}{2}$. Let $p \in \Z$ be such that $\vert (\xi k - r) - p \vert \leqslant \frac{1}{2}$ (we have $p = \lfloor \xi k - r \rfloor$, or  $p = \lfloor \xi k - r \rfloor +1$).

Using these estimates, we get
\begin{eqnarray*}
|\beta^k -\lambda| & = & |e^{2i\pi\xi k}-e^{2i\pi r}|= |e^{2i\pi\xi k}-e^{2i\pi( r +p)}|\\
   &=  & |e^{2i\pi(\xi k-r-p)}-1|=  |e^{i\pi(\xi k-r-p)}-e^{-i\pi(\xi k-r-p)}|\\
     & = & 2|\sin (\pi (\xi k-r-p))|\\
     & \geq & 2\frac{2}{\pi}\pi |\xi k -r-p|
\end{eqnarray*}
since $p \in \Z$ is chosen such that $\vert \xi k - r - p \vert \leqslant \frac{1}{2}$.
Consequently,
\begin{eqnarray*}
\frac{1}{|\beta^k-\lambda|} & \leq & \frac{1}{4}\frac{1}{|\xi k-r-p|} 
  =  	\frac{1}{4k}\frac{1}{\left| \xi -\frac{p_0 +pq_0}{q_0 k} \right|}\\
  & \leq & \frac{1}{4k} \frac{1}{\gamma} \frac{1}{(q_0k)^{-\tau}}=ck^{\tau-1},
\end{eqnarray*}
where $c:=\frac{1}{4\gamma} q_0^\tau$. 

Thus 
\[\limsup_{k\to\infty}\frac{1}{|\beta^k -\lambda|^{1/k}}\leq \limsup_{k\to\infty}c^{1/k}k^{(\tau -1)/k}\leq 1.\]
Let $g(z)=\sum_{n=0}^\infty a_n z^n$ such that $g\in\Hol (\D)$. Then $f(z)=\sum_{n\geq 0}^\infty \frac{a_n}{\beta^n-\lambda}z^n$ defines $f\in\Hol(\D)$ and $f\circ \varphi (z)-\lambda f(z)=g(z)$ for all $z\in\D$. 
Thus $C_\varphi -\lambda\Id$ is surjective. From Proposition~\ref{prop:4.1} we know that  $C_\varphi -\lambda\Id$ is
injective. The theorem is proved. 
\end{proof}	
Thus, in the situation of Theorem~\ref{th:4.2}, for $\beta=e^{2i\pi\xi}$, 
\[   \{  \beta^n \ : \ n\in\N_0\}\subset \sigma(C_\varphi) \subset \T\setminus \{ \lambda\neq 1:\exists r\in\Q,\lambda=e^{2i\pi r}   \}.  \]
In particular, the closure of $\sigma(C_\varphi)$ is $\T$ and the closure of $\T\setminus \sigma(C_\varphi)$ is also $\T$. 
\begin{rem}\label{rem:4.3}
	The space $\Hol_0(\D):=\{   f\in\Hol(\D):f(0)=0\}$ is invariant under $C_\varphi$. Bonet considered the restriction $T_0$ 
	 of $C_\varphi$ to $\Hol_0(\D)$ and showed that under the hypothesis of Theorem~\ref{th:4.2}, $1\not\in \sigma(T_0)$ (see \cite[Corollary 3]{Bonet}). By Lemma~\ref{lem:2.3} this implies that 
	 \begin{equation}\label{eq:4.1}
	 \overline{\beta}\not\in \sigma(C_\varphi).
	 \end{equation} 
\end{rem}
\begin{rem}\label{rem:4.4}
Let $\xi\in D(\tau)$, $\beta=e^{2i\pi\xi}$. Let $r_0\in\Q$ such that $\lambda_0:=e^{2i\pi r_0}\neq 1$. Then, by Theorem~\ref{th:4.2}, $\lambda_0\in\rho(C_\varphi)$. But the resolvent is not strongly continuous at $\lambda_0$. In fact, if $\lambda_n\not\in\T$ and $\lambda_n\to\lambda_0$ then for all $0<r<1$ there exists $f\in A(\D)$ such that 
\[   \sup_{n\in\N} \|R_{\lambda_n} f\|_{A(r\D)}=\infty. \]
Here $A(r\D)={\mathcal C}(r\overline{\D})\cap \Hol(r\D)$ and for $g\in \Hol(\D)$, $\|g\|_{A(r\D)}=\sup_{|z|\leq r}|g(z)|$.    
\end{rem}
\begin{proof}
Assume that there exists $r\in (0,1)$ with $\sup_{n\in\N} \|R_{\lambda_n}f\|_{A(r\D)}<\infty$ for all $f\in A(\D)$. 
Then by the uniform boundedness principle there exists $c>0$ such that 
\[  \| R_{\lambda_n}f\|_{A(r\D)}\leq c\|f\|_{A(\D)}  \]
for all $f\in A(\D)$ and all $n\in\N$. Choose $f=e_k$. Since $C_\varphi e_k=\beta^ke_k$, $R_{\lambda_n}e_k =\frac{1}{\lambda_n -\beta}e_k$. Hence 
\[    r\frac{1}{|\lambda_n -\beta^k|}=\| R_{\lambda_n} e_k\|_{A(r\D)}\leq c\|e_k\|_{A(\D)} =c\]
for all $n\in\N$ and  $k\in\N$. Hence 
\[  \frac{1}{|\lambda_n-\lambda_0|}\leq \sup_{|w|=1}\frac{1}{|\lambda_n-w|} =\sup_k \frac{1}{|\lambda_n-\beta^k|} \leq \frac{c}{r}.   \]
This is a contradiction since $|\lambda_n-\lambda_0|\to 0$ as $n\to\infty$. 
\end{proof}	
The phenomenon expressed in Theorem~\ref{th:4.2} persists even for a larger class of $\xi$ than the diophantine numbers, considered for instance by Lang in \cite{Lang71}.
\begin{thm}\label{th:4.5}
    Let $g~:~\Z \rightarrow \R^+$ such that $\forall q_0 \in \N$, $\liminf_{k \to \infty}g(q_0 k)^{1/k}\geqslant 1 $. Let $\beta=e^{2i\pi\xi}$, where $|\xi-p/q|\geq g(q)\geq 0$ for all $p\in\Z$, $q\in\N$.
	 Then, 
	 \[ \left\{ \beta^n : n \in \N_0 \right\} \subset \sigma(C_\varphi) \subset \left\{ e^{2i\pi x} : x \notin \Q \right\} \cup \left\{ 1 \right\} . \] 
\end{thm}
This result is a generalization of Theorem \ref{th:4.2} : for $\tau > 2$, $\gamma >0$, and $g(q) = \gamma q^{- \tau}$, one gets the result stated in Theorem \ref{th:4.2}.
\begin{proof}
Thanks to Propositions \ref{prop:4.1} and \ref{prop:4.1prime}, it suffices to show that if $\lambda = e^{2i\pi r}$ with $r \in \Q$ and $\lambda \neq 1$, then $\lambda \in \rho(C_\varphi)$. Let $\lambda= e^{2i\pi r}$ be such a complex number. Let $p_0 \in \Z$ and $q_0 \in \N$ such that $r = \frac{p_0}{q_0}$. Let $k \in \N$. Let $p \in \Z$ be such that $\left\vert (\xi k - r) - p \right\vert \leqslant \frac{1}{2}$. Then, with the same computation as in the proof of Theorem \ref{th:4.2}, we get
\[ \left\vert \beta^k - \lambda \right\vert \geqslant 4 \left\vert \xi k - r - p \right\vert .  \]
Hence,
\begin{eqnarray*}
\frac{1}{\left\vert \beta^k - \lambda \right\vert} & \leqslant & \frac{1}{4k} \times \frac{1}{\left\vert \xi - \frac{p_0 + q_0 p }{q_0k}\right\vert} \\
& \leqslant & \frac{1}{4k} \times \frac{1}{g(q_0k)}.
\end{eqnarray*}
Thus,
\[ \left( \frac{1}{\left\vert \beta^k - \lambda \right\vert} \right)^{\frac{1}{k}}  \leqslant \frac{1}{4^{\frac{1}{k}}k^{\frac{1}{k}}} \times  \frac{1}{g(q_0k)^{\frac{1}{k}}},  \]
and so 
\[ \limsup_{k \to \infty}  \left( \frac{1}{\left\vert \beta^k - \lambda \right\vert} \right)^{\frac{1}{k}} \leqslant 1 \times \frac{1}{\liminf_{k \to \infty} g(q_0k)^{\frac{1}{k}}} \leqslant 1.\]
The conclusion of the proof is then identical to the proof of Theorem \ref{th:4.2}.
\end{proof}	
\section{The spectrum of weighted composition operators induced by  aperiodic rotations in Banach spaces}\label{sec:5}
We interrupt our study of  weighted composition operators induced by rotations on $\Hol(\D)$ and consider Banach spaces of holomorphic functions. One reason is that we will use the results on the disc algebra 
\[  A(\D)  :={\mathcal C}(\overline{\D})\cap \Hol(\D)   \] 
when we consider $\Hol(\D)$ again in Section~\ref{sec:7}. The other reason is that the technique we use 
(i.e. spectral decomposition) is simpler for Banach spaces. 

Let $X$ be a Banach space  such that 
   \[X\hookrightarrow \Hol(\D),\]
   i.e. $X$ is a  subspace of $\Hol(\D)$ and the injection is continuous, where $\Hol(\D)$ carries the topology of uniform convergence on compact subsets of $\D$. We let $e_n(z)=z^n$ for $n\in\N_0$ and $z\in\D$. We assume that 
 \begin{equation}\label{eq:5.1}
    e_0\in X;
 \end{equation}   
 \begin{equation}\label{eq:5.2}
    \mbox{if } f,f_1\in\Hol(\D)\mbox{ are such that } f=e_1f_1,\mbox{ then }f\in X\Longleftrightarrow f_1\in X;  
 \end{equation}   
   \begin{equation}\label{eq:5.3}
  \mbox{if } f\in X, \mbox{ then }z\mapsto f(\beta z)\in X\mbox{ for all }\beta\in\T. 
  \end{equation}     
  It follows from (\ref{eq:5.1}) and (\ref{eq:5.2}) that $\C[z]\subset X$. Moreover, for each $\beta\in\T$, $(U_\beta f)(z)=f(\beta z),z\in\D$ defines an isomorphism $U_\beta$ on $X$. 
  
 Let $m\in\Hol(\D)$ such that 
  \begin{equation}\label{eq:5.4}
 mf\in X\mbox{ for all }f\in X.
 \end{equation}   
 Then, given $\beta\in\T$, we consider the operator $T:\Hol(\D)\to\Hol(\D)$ given by 
 \[  (Tf)(z)=m(z)f(\beta z).\]
It leaves $X$ invariant and we denote by $T_X$ the restriction of $T$ to $X$. Then $T_X\in{\mathcal L}(X)$ by the closed graph theorem. We are mainly interested in the \emph{aperiodic case}; i.e. when $\beta$ is not a root  of the unity. Then the spectrum of $T_X$ is rotationally invariant as we prove in Proposition~\ref{prop:5.2}. 

We start by a technical lemma which will be used later. Recall from Remark~\ref{rem:4.3} that it may happen that $\lambda\in\rho(T)$ but $\lambda \beta\in \sigma_p(T)$. 
\begin{lem}\label{lem:5.1}
Let $\mu\in\C\setminus \{ \overline{\beta} m(0)\}$, $\mu\neq 0$. 
\begin{itemize}
	\item[a)] If $\mu\Id -T$ is surjective, then $\beta\mu\Id -T$ is also surjective. 
	\item[b)] If $\mu\in\rho(T)$ and $\beta\mu\not\in\sigma_p(T)$, then $R_{\beta\mu}$ can be obtained from $R_\mu$ by the following:\\
	for $g\in\Hol(\D)$ of the form $g(z)=zg_1(z)$ with $g_1\in\Hol(\D)$ one has 
	\begin{equation}\label{eq:5.5}
	R_{\beta \mu}g=e_1 R_\mu (\overline{\beta} g_1);	  
	\end{equation}	  
	\begin{equation}\label{eq:5.6}
	R_{\beta \mu}e_0=  \frac{1}{\beta \mu -m(0)}e_0  + R_{\beta \mu}\left(  \frac{m-m(0)}{\beta \mu - m(0)}\right) ;  
	\end{equation}
	 \item[c)] If $\mu\in \rho(T_X)$ and $\beta \mu\not\in \sigma_p(T_X)$ then $\beta\mu\in \rho(T_X)$ and $R_{\beta \mu}$ is given by (\ref{eq:5.5}) and (\ref{eq:5.6}). 
\end{itemize}	
\end{lem}     
\begin{proof}
Let $\mu\in\C$ such that $\mu\neq 0$ and $\beta\mu\neq m(0)$. \\
a) Assume that $(\mu\Id -T)$ is surjective. \\
\emph{First case}: Let $g\in\Hol(\D)$ such that $g(0)=0$. Then there exists $g_1\in\Hol(\D)$ such that $g(z)=zg_1(z)$, $z\in\D$. By hypothesis there exists $f_1\in\Hol(\D)$ such that $(\mu\Id -T)f_1=\overline{\beta} g_1$. Let $f(z)=zf_1(z)$. Then 
\begin{eqnarray*}
((\beta\mu \Id -T)f)(z) & = & \beta\mu zf_1(z)-m(z)\beta z f_1(\beta z)\\
  &  = & \beta z ((\mu\Id -T)f_1)(z)=zg_1(z)=g(z).  	
\end{eqnarray*}	
\emph{Second case}: Let $g=e_0$. By the first case there exists $\tilde{f} \in\Hol(\D)$ such that 
  \[  (\beta\mu \Id -T)\tilde{f} =\frac{m-m(0)}{\beta \mu - m(0)} \mbox{ and }\tilde{f}(0)=0.   \]
Let $f(z)=\frac{1}{\beta \mu -m(0)} + \tilde{f}(z)$. Then $f(0)=  \frac{1}{\beta \mu -m(0)} $ and consequently, 
\begin{eqnarray*}
(\beta\mu\Id -T)f & = & (\beta\mu -m)f(0)+(\beta\mu\Id -T)\tilde{f} \\
  & = & \beta\mu f(0) - mf(0) +(m-m(0))f(0)=1. 	
\end{eqnarray*}
Since $\Hol(\D)=\Hol_0(\D)\oplus \C e_0$ with $\Hol_0(\D):=\{  f\in\Hol(\D):f(0)=0 \}$, a) and b) are proved. \\
c) By our hypothesis on $X$ we have $X=X_0\oplus \C e_0$ where $X_0=\{ f\in X:f(0)=0 \}$. Moreover, since $mX\subset X$, $m=me_0\in X$. Now the above proof also gives the statement of c).   
\end{proof}	
Now we show rotational invariance of the spectrum of $T_X$. We let $\T=\partial \D=\{  \lambda \in\C :|\lambda |=1\}$.
\begin{prop}\label{prop:5.2}
Assume that $|\beta|=1$ and $\beta^n\neq 1$ for all $n\in\N$. Then 
\begin{itemize}
	\item[a)] $m(0)\T\subset \sigma(T_X)$.
	\item[b)] $\lambda\in\sigma(T_X)$  implies $\lambda\T\subset \sigma(T_X)$.  
\end{itemize}	
\end{prop}  
\begin{proof}
a) It follows from Lemma~\ref{lem:2.5} that $\beta^k m(0)\Id -T_X$ is not surjective. Since $\sigma(T_X)$ is closed, it follows that $\T m(0)\subset \sigma(T_X)$. \\
b) Let $\lambda\in\sigma(T_X)$, $|\lambda|\neq |m(0)|$. \\
\emph{First case}: $\lambda\in\sigma_p(T)$. Then $\beta^k\lambda \in\sigma_p(T_X)$ for all $k\in\N_0$ by the proof of Proposition~\ref{prop:3.1}. Since $\sigma(T_X)$ is closed, this implies that $\lambda \T\subset \sigma(T_X)$. \\
\emph{Second case}: $(\lambda\Id -T_X)$ is not surjective. Applying Lemma~\ref{lem:5.1} to $\mu=\overline{\beta}\lambda$, we deduce that $(\overline{\beta} \lambda -T_X)$ is not surjective either. Thus, by iteration, $\overline{\beta}^k \lambda \in\sigma(T_X)$ for all $k\in\N$. This implies that $\lambda\T\subset\sigma(T_X)$.    
\end{proof}	
Next we show that the spectrum of $T_X$ is a disc whenever $m$ has a zero in $\D$. By $r_\sigma (T_X)$ we denote the spectral radius of $T_X$. 
\begin{thm}\label{th:5.3}
	Assume that $|\beta|=1$, $\beta^n\neq 1$ for all $n\in\N$ and that there exists $z_0\in\D$ such that $m(z_0)=0$. Then \[\sigma(T_X)=\{  \lambda\in\C  : |\lambda|\leq r_\sigma(T_X) \} .\] 
\end{thm}  
\begin{proof}
Since $(T_X f)(z_0)=0$ for all $f\in X$, $T$ is not invertible and so $0\in\sigma(T_X)$. Assume that there exists $\lambda_0\in\rho(T_X)$ such that \[0<r_0:=|\lambda_0|<r_\sigma (T_X).\]
Then, by Proposition~\ref{prop:5.2}, $r_0\T\subset \rho(T_X)$. Thus, by the spectral decomposition theorem \cite[VII Theorem 19]{DS88}, there exist closed subspaces $X_j$ of $X$ such that $TX_j\subset X_j$, $j=1,2$, $X=X_1\oplus X_2$ and for $T_j:=T_{|X_j}$, 
\[\sigma(T_1)=\{   \lambda\in\sigma(T_X):|\lambda|<r_0  \} \mbox{ and }   \sigma(T_2)=\{   \lambda\in\sigma(T_X):|\lambda|>r_0  \}.\]
Since $\sigma(T_X)\cap \{   \lambda\in\C :|\lambda|=r_\sigma (T_X)\} \neq \emptyset$, one has $X_2\neq \{0\}$. The operator $T_2$ is invertible. Let $0\neq f\in X_2$. Then for each $n\in\N$ there exists $g_n\in X_2$ such that $f=T_2^n g_n$. Thus 
\begin{equation}\label{eq:5.7}
   f(z)=m(z)m(\beta z)\cdots m(\beta^{n-1}z) g_n(\beta^n z).
 \end{equation}  	
 \emph{First case}: $z_0\neq 0$. It follows from (\ref{eq:5.7}) that $f(\overline{\beta}^{n-1}z_0)=0$. Since $n$ is arbitrary, $f$ vanishes on $z_0\T$ and hence $f=0$, a contradiction. \\
 \emph{Second case}: $z_0= 0$. Then there exists $m_1\in\Hol(\D)$ such that $m(z)=zm_1(z)$. It follows from (\ref{eq:5.7}) that $0$ is a zero of $f$ of order $n$. Since $n$ is arbitrary $f=0$, a contradiction. 
 
 We have shown that $D(0,r_\sigma(T_X))\subset \sigma(T_X)$. Since $\sigma(T_X)$ is closed, the claim follows.    
\end{proof}	
We consider three examples:  
the \emph{disc algebra} $A(\D):={\mathcal C}(\overline{\D})\cap \Hol(\D)$, the \emph{Wiener algebra} $W^+(\D):=\{ f\in\Hol(\D) :\sum_{n=0}^\infty|\widehat{f}(n)|<\infty\}$, where for $f(z)=\sum_{n=0}^\infty a_n z^n$, $f\in\Hol(\D)$, we let $\widehat{f}(n):=a_n$, and the \emph{Hardy spaces} $H^p(\D)$, $1\leq p\leq \infty$.   Recall that $H^\infty(\D)$ is the algebra of all bounded holomorphic functions on $\D$ equipped with the sup norm whereas for $1\leq p<\infty$, 
\[ H^p(\D):=\{  f\in\Hol(\D): \|f\|^p_p:=\sup_{0<r<1}  \frac{1}{2\pi}\int_0^{2\pi}|f(re^{it})|^p dt<\infty\}.\]

They all satisfy (\ref{eq:5.1})-(\ref{eq:5.3}). It follows that $U_\beta$ is an isometry on $X$. 

Let $m\in\Hol(\D)$. Then $mX\subset X $ if and only if \\ 

\hspace{2cm} $m\in A(\D)$ for $X=A(\D)$,  \\

 \hspace{2cm}  $m\in W^+(\D)$ for $X=W^+(\D)$ and \\
 
\hspace{2cm}  $m\in H^\infty (\D)$ for $X=H^p(\D)$ if $1\leq p\leq \infty$. \\

In the first two cases we have $X\subset {\mathcal C}(\overline{\D})$ and the proof of Theorem~\ref{th:5.3} remains true whenever $m$ has a $0$ in $\overline{\D}$ (and not merely in $\D$). We state this as a corollary (from the proof). 
\begin{cor}\label{cor:5.4}
Let $X=A(\D)$ and $m\in A(\D)$, or $X=W^+(\D)$ and $m\in W^+(\D)$. Let $\beta\in\C$ such that $|\beta|=1$, $\beta^n\neq 1$ for all $n\in\N$. If there exists $z_0\in\overline{\D}$ such that $m(z_0)=0$, then 
\[  \sigma(T_X)=\{  \lambda\in\C :|\lambda|\leq r_\sigma(T_X)  \}.   \] 
\end{cor} 
For $X=A(\D)$, Corollary~\ref{cor:5.4} is given by Kamowitz \cite[Theorem 4.8 and Theorem 4.9]{Ka78} whose proof uses a change of limits (7 lines after (4.2)) and arguments concerning negligible sets which seem difficult to be justified. Our proof is very different. Note that $T^n$ is given by 
\[ (T^n f)(z)=m_n(z)f(\beta^n z),\,\,\,\, (f\in\Hol(\D))    \]
where $m_n(z)=m(z)m(\beta z)\cdots m(\beta^{n-1}z)$. This allows one to calculate, or at least estimate, the spectral radius $r_\sigma (T_X)$ for $X=A(\D),W^+(\D)$ or $H^p(\D)$. In fact 
\[r_\sigma (T_X)=\lim_{n\to\infty}\|T_X^n\|^{1/n}   =   \lim_{n\to\infty} \sup_{z\in \D}|m_n(z)|^{1/n}.\]
 In the case of $X=A(\D)$ or $X=W^+(\D)\subset A(\D)$, this expression can be computed. The following lemma is due to Kamowitz \cite[Lemma 4.2 and 4.4]{Ka78}.
\begin{lem}
\label{lem:5.5}
 Let $m\in A(\D)$, $\beta\in\C$, $|\beta|=1$, $\beta^n\neq 1$ for all $n\in\N$.
 \begin{itemize}
 	\item[a)] Then 
 	\[  r_\sigma(T_{A(\D)})=\lim_{n\to\infty} \sup_{|z|\leq 1}|m_n(z)|^{1/n} =\exp\left(  \frac{1}{2\pi} \int_0^{2\pi}\log|m(e^{it})|dt\right) =:M^*.  \] 
 	\item[b)] If $m(z)\neq 0$ for all $|z|=1$, then
 	\[  \lim_{n\to\infty}|m_n(z)|^{1/n}=M^*\mbox{ uniformly in }z\in\T.\]
 \end{itemize} 
\end{lem} 	 
If $m\in H^\infty (\D)$, then we set 
\[   M^* :=\exp\left(  \frac{1}{2\pi} \int_0^{2\pi}\log |m^*(e^{it})|dt\right) <\infty ,  \]
where $m^*$ denotes the radial limit of $m$. In that case we merely have a lower estimate of the spectral radius. 
\begin{lem}\label{lem:5.6}
Let $m\in H^\infty(\D)$, $|\beta|=1$, $\beta^n\neq 1$ for all $n\in\N$, and $X=H^p(\D)$ with $1\leq p\leq \infty$. 
Then \[  r_\sigma (T_X) \geq M^*.    \]  
\end{lem}
\begin{proof}
One has 
\[ \sup_{|z|< 1} |m_n(z)|^{1/n}=\mathop{ess}\sup_{\theta\in\R}|m_n^*(e^{i\theta})|^{1/n}    \]
since $\|f\|_{H^\infty}	=\|f^*\|_{L^\infty (\T)}$. But 
\begin{eqnarray*}
	 \log|m_n^*(e^{i\theta})|^{1/n} & = & \frac{1}{n}(\log |m^*(e^{i\theta})| + \log |m^*(\beta e^{i\theta})|+\cdots +   \log |m^*(\beta^{n-1} e^{i\theta})|)  \\
	  &  & \to \frac{1}{2\pi} \int_0^{2\pi} \log|  m^*(e^{it})|dt  \,\, \mbox{  $\theta$-a.e. as }n\to\infty
	\end{eqnarray*}
by the pointwise ergodic theorem.  This implies the claim.  
\end{proof}	
In any case, for $m\in A(\D)$, by Lemma~\ref{lem:5.5}, 
\begin{equation}\label{eq:5.8}
r_\sigma (T_{A(\D)}) =r_\sigma (T_{H^p(\D)});
\end{equation} 
and for $m\in W^+(\D)\subset A(\D)$, 
\begin{equation}\label{eq:5.9}
r_\sigma (T_{W^+(\D)}) =r_\sigma (T_{A(\D)})=r_\sigma (T_{H^p(\D)}).
\end{equation} 
For $m\in A(\D)$, the operator $T_{A(\D)}$ is invertible if and only if $m(z)\neq 0$ for all $z\in \overline{\D}$. In that case $\frac{1}{m}\in A(\D)$ and 
\[     (T^{-1}f)(z)=\frac{1}{m(\overline{\beta} z)}f(\overline{\beta}z).      \] 
If $m\in W^+(\D)$ such that $m(z)\neq 0$ for all $z\in \overline{\D}$, then by Wiener's theorem, 
$\frac{1}{m}\in W^+(\D)$. Thus  $T_{W^+(\D)}$ is invertible (see for instance \cite[Chap.18, exercise 8]{Ru74}).

Note that for $m\in A(\D)$ such that $m(z)\neq 0$ for all $z\in \overline{\D}$, 
\begin{equation}\label{eq:5.10}
\exp\left( \frac{1}{2\pi}\int_0^{2\pi}\log |m(e^{it})|dt\right)  =|m(0)|,
\end{equation}
(see \cite[15.18]{Ru74}). 
From this we obtain the following result (due to Kamowitz \cite[Theorem 4.7]{Ka78} in the case $X=A(\D)$) whose short proof is repeated here. 
\begin{prop}\label{prop:5.7}
  Let 
  \begin{itemize}
      \item[a)] 
  $X=A(\D)$ or $H^p(\D)$ ($1\leq p\leq \infty$) and $m\in A(\D)$; or
    \item[b)]$X=W^+(\D)$ and $m\in W^+(\D)$. 
\end{itemize}    
  Assume that $\vert \beta \vert = 1$, $\beta^n \neq 1$ for all $n \in \N$, and that $m(z)\neq 0$ for all $z\in\overline{\D}$. Then 
  \[  \sigma(T_X)=\{  \lambda \in\C :|\lambda|=|m(0)|  \}.   \]
 \end{prop} 
\begin{proof}
By (\ref{eq:5.10}) and Lemma~\ref{eq:5.5} we have \[r_\sigma (T_X) =|m(0)| \mbox { and } r_\sigma(T^{-1})=\frac{1}{|m(0)|}.\] 	
If $\lambda\in \sigma(T_X)$ then $\frac{1}{|\lambda|}\leq r_\sigma (T_X^{-1})= \frac{1}{|m(0)|}$. Thus 
\[  |m(0)|\leq |\lambda|\leq r_\sigma (T_X)=|m(0)|.    \]
Now the claim follows from Corollary~\ref{cor:2.2}. 
\end{proof}
\section{ Spectral decomposition and Waelbroeck spectrum}\label{sec:6}
We now continue our study on $\Hol(\D)$. It is a Fr\'echet space for the topology of compact convergence which is defined by the seminorms 
\[ \|f\|_{A(r\D)}:=\sup_{|z|\leq r}|f(z)|, \,\,\, f\in \Hol(\D)  \]
for $r\in (0,1)$. Thus $f_n\to f$ in $\Hol(\D)$ if and only if $\|f_n-f\|_{A(r\D)}\to 0$ for all $0<r<1$. 

Let $T:\Hol(\D)\to \Hol(\D)$ be linear and continuous. If $\lambda\in\rho(T)$, then $R_\lambda=(\lambda\Id -T)^{-1}$ is continuous (by the closed graph theorem). This means the following:\\
for all $r\in (0,1)$ there exist $r'\in (0,1)$ and $c>0$ such that 
\begin{equation}\label{eq:6.1}
\|  R_\lambda f\|_{A(r\D)}\leq c\|f\|_{A(r'\D)}\mbox{ for all }f\in\Hol(\D).
\end{equation}  
We have seen that the spectrum $\sigma(T)$ of $T$ is not closed, in general, and the resolvent may not be strongly continuous on $\rho(T)$, see Section~\ref{sec:4}.  For that reason we will also consider the \emph{Waelbroeck spectrum}. Here we define the \emph{Waelbroeck resolvent set} by 
\begin{eqnarray*}
\rho_W(T) & = & \{  \lambda\in\rho(T): \exists \delta >0\mbox{ such that }\overline{D}(\lambda,\delta) \subset \rho(T)\mbox{ and } \\
 &  & \sup_{|\lambda-\mu|\leq \delta} \|R_\mu f\|_{A(r\D)}<\infty\mbox{ for all }f\in\Hol(\D)\mbox{ and }r<1  \}	
	\end{eqnarray*}   
Then the \emph{Waelbroeck spectrum} is by definition $\sigma_W(T)=\C\setminus \rho_W(T)$. 

It is clear that $\sigma_W(T)$ is an open subset of $\C$. One may express $\rho_W(T)$ also by an \emph{equicontinuity property}, i.e. (\ref{eq:6.1}) is asked to hold uniformly with respect to $\lambda$ (see (\ref{eq:6.2})). 
\begin{lem}\label{lem:6.1}
Let $\Lambda\subset \rho_W(T)$ be compact. Then for every $0<r<1$ there exist $0<r'<1$ and $c\geq 0$ such that 
\begin{equation}\label{eq:6.2}
\|R_\lambda f\|_{A(r\D)}\leq c\|f\|_{A(r'\D)}
\end{equation}  
for all $f\in \Hol(\D)$ and all $\lambda\in\Lambda$. 
\end{lem} 
\begin{proof}
Let $0<r<1$. By the compactness of $\Lambda$ and the definition of $\rho_W(T)$, 
\[\sup_{\lambda\in\Lambda} \|R_\lambda f\|_{A(r\D)}<\infty \mbox{ for all }f\in \Hol(\D) \mbox{ and } r \in (0;1).\]
Now one proceeds as for the proof of the uniform boundedness principle. For $n \in \N$, let
\[   X_n:=\{  f\in\Hol(\D):   \|R_\lambda f\|_{A(r\D)}\leq n\,\,\,  \forall \lambda\in\Lambda  \}.\]    	
Since $X_n$ is closed and $\Hol(\D)=\cup_{n\in\N} X_n$, by Baire's theorem, there exists $n_0\in\N$ such that the interior of 
$X_{n_0}$ is nonempty. Thus there exist $f_0\in X_{n_0}$, $r'\in (0,1)$, $\varepsilon>0$ such that 
\[  {\mathcal U} :=\{  f\in\Hol(\D): \|f-f_0\|_{A(r'\D)}\leq \varepsilon \}\subset X_{n_0}.     \]
Let $g\in\Hol(\D)$ such that $\|g\|_{A(r'\D)}\leq 1$. Then $f_0 +\varepsilon g\in{\mathcal U}$. 
Hence 
\[  \varepsilon \|R_\lambda g\|_{A(r\D)}\leq \|  R_\lambda (f_0 +\varepsilon g)\|_{A(r\D)} +\|R_\lambda f_0\|_{A(r\D)}\leq 2n_0\mbox{ for all }\lambda\in\Lambda.    \] 
Thus $\|R_\lambda f\|_{A(r\D)}\leq \frac{2n_0}{\varepsilon}\|f\|_{A(r'\D)}$ for all $f\in \Hol(\D)$ and $\lambda\in\Lambda$. 
\end{proof}	
If $\Omega\subset \C$ is open, a function $F:\Omega\to \Hol(\D)$ is called \emph{holomorphic}  if 
\[   \lim_{\lambda \to\lambda_0} \frac{F(\lambda)-F(\lambda_0)}{\lambda-\lambda_0}=:F'(\lambda_0)\]
exists in $\Hol(\D)$ for all $\lambda_0\in\Omega$. 
\begin{lem}\label{lem:6.2}
Let $f\in\Hol(\D)$. The function 
\[\lambda\mapsto R_\lambda f:\rho_W(T)\to\Hol(\D)\]
 is holomorphic and thus continuous. In particular, for each $z\in\D$ the function  $\lambda\mapsto (R_\lambda f)(z)$ is holomorphic on $\rho_W(T)$.    
\end{lem}
\begin{proof}
Let $\lambda_0\in\rho_W(T)$. Choose $\delta>0$ such that $\overline{D}(\lambda_0,\delta)\subset \rho_W(T)$. Let $0<r<1$ be arbitrary. By Lemma~\ref{lem:6.1} there exist $c>0$, $0<r'<1$ such that 
\[   \|R_\lambda f\|_{A(r\D)}\leq c\|f\|_{A(r'\D)}  \]
for all $f\in\Hol(\D)$ and $\lambda\in \overline{D}(\lambda_0,\delta)$.  	
Thus, by the resolvent identity for $\lambda\in \overline{D}(\lambda_0,\delta)$, 
\begin{eqnarray*}
\| R_\lambda f-R_{\lambda_0}f\|_{A(r\D)} & = & |\lambda -\lambda_0| \|R_\lambda R_{\lambda_0} f\|_{A(r\D)}\\
  & \leq & |\lambda-\lambda_0|c\|R_{\lambda_0} f\|_{A(r'\D)}	
\end{eqnarray*}
Hence $R_\lambda f\to R_{\lambda_0}f$ in $\Hol(\D)$ as $\lambda\to\lambda_0$. In particular
\[  \frac{R_\lambda f-R_{\lambda_0}f}{\lambda-\lambda_0}=-R_\lambda R_{\lambda_0}f\to -R_{\lambda_0}^2 f    \]
in $\Hol(\D)$ as $\lambda\to\lambda_0$. 
\end{proof}	
We will need the following spectral decomposition which is well-known in the Banach space case. To be complete we give a proof in this special situation. For  a much more general result  in Fr\'echet spaces we refer to 
Th\'eor\`eme III. 3.11   in the monograph \cite{Va82} by Vasilescu. 
\begin{thm}\label{th:6.3}[Spectral decomposition]
 Let $T\in{\mathcal L}(\Hol(\D))$ and $r_0\in (0,\infty)$ such that $r_0\T\subset \rho_W(T)$. Then there exist closed subspaces  $X_1,X_2$ of $\Hol(\D)$ such that  $TX_j\subset X_j$, $j=1,2$, $\Hol(\D)=X_1\oplus X_2$ and for $T_j:=T_{|X_j}$, 
 \[\sigma_W(T_1)=\{   \lambda\in\sigma_W(T_X):|\lambda|<r_0  \} \mbox{ and }   \sigma_W(T_2)=\{   \lambda\in\sigma_W(T_X):|\lambda|>r_0  \}.\]
\end{thm}
For the proof we need two lemmas.
\begin{lem}\label{lem:6.4}
Let $\Lambda\subset\C$ and let $g:\Lambda\times \D\to\C$ such that $g(t,\cdot)\in\Hol(\D)$ for all $t\in\Lambda$. Then $g$ is continuous if and only if 
\[   t\mapsto   g(t,\cdot):\Lambda\to \Hol(\D) \]
is continuous.  	
\end{lem}
\begin{proof}
$\Rightarrow$ Let $t_n\to t_0$, $K\subset \D$ compact. Assume that $g(t_n,z)$ does not converge to $g(t_0,z)$ uniformly on $K$.  Then, passing to a subsequence we find $\varepsilon>0$, $z_n\in K$ such that $|g(t_n,z_n)-g(t_0,z_n)|\geq \varepsilon$. We may also assume that $z_n\to z_0$. Thus $g$ is not continuous at $(t_0,z_0)$. \\
$\Leftarrow$ If $t_n\to t_0$ in $\Lambda$ and $z_n\to z_0$ in $\D$, then 
\[|g(t_n,z_n)-g(t_0,z_0)|\leq |g(t_n,z_n)-g(t_0,z_n)| +|g(t_0,z_n)-g(t_0,z_0)|\to 0  \] 
as $n\to\infty$ by hypothesis.    	
\end{proof}	
\begin{rem}\label{rem:6.5}
Lemma~\ref{lem:6.4} allows us to describe $\rho_W(T)$ yet in a different way. It is the largest open set $\Omega$ in $\rho(T)$ such that the mapping 
\[     (\lambda,z)\mapsto (R_\lambda f)(z):\Omega\times \D\to\C   \] 
is continuous for all $f\in\Hol(\D)$. 
\end{rem}
Let $g:[0,2\pi]\to \Hol(\D)$ be continuous. Then we define the Riemann integral of $g$ as 
\begin{equation}\label{eq:6.3}
\int_0^{2\pi} g(t)dt =\lim_{n\to\infty}S(\Pi_n,g).
\end{equation} 
Here, $(\Pi_n)_n$ is a sequence of partitions such that %
$\lim_{n\to\infty}
\delta_{\Pi_n}=0$.  For an arbitrary partition $\Pi=\{ 0=t_0<t_1<\cdots < t_m=2\pi \}$, the Riemann sum is defined as $S(\Pi ,g)=\sum_{i=1}^m g(t_{i-1})(t_i-t_{i-1})$ and $\delta_{\Pi}:=\max_{1\leq i\leq m}(t_{i}-t_{i-1})$.
Since  $g(t)_{|r\D}\in A(r\D)$, the convergence of (\ref{eq:6.3}) follows from the corresponding result in the Banach space $A(r\D)$. 
\begin{lem}\label{lem:6.6}
Let $g:[0,2\pi]\to\Hol(\D)$ be continuous and $R\in {\mathcal L}(\Hol(\D))$. Then 
\[R\int_0^{2\pi} g(t)dt =\int_0^{2\pi} Rg(t)dt.\]
\end{lem} 
\begin{proof}
One has $RS(\Pi,g)=S(\Pi,Rg)$. The result follows from (\ref{eq:6.3}).
\end{proof}	
The same argument shows that 
\begin{equation}\label{eq:6.4}
\int_0^{2\pi} g(t)dt (z)=\int_0^{2\pi} g(t)(z)dt \,\,\, \mbox{ for all }z\in\D.
\end{equation}
\begin{proof}[Proof of Theorem~\ref{th:6.3}]
There exists $r_0<r_1$ such that $\lambda\in\rho_W(T)$ whenever $r_0\leq |\lambda|\leq r_1$. Define $P\in {\mathcal L}(\Hol(\D))$ by 
\[    (Pf)(z)=\frac{1}{2i\pi}\int_{|\lambda|=r_0} (R_\lambda f)(z)d\lambda     \]
for all $z\in\D$, $f\in\Hol(\D)$.  We show that $P^2=P$. Note that 
\[  Pf (z)=\frac{1}{2i\pi}\int_{  |\mu|=r_1} (R_\mu f)(z)d\mu  \]
by Cauchy's theorem. Moreover, using the resolvent identity, Cauchy's integral formula and Lemma~\ref{lem:6.6}, we obtain  
\begin{eqnarray*}
P(Pf)(z)& = & \frac{1}{2i\pi}\int_{|\lambda|=r_0}(R_\lambda (Pf))(z)d\lambda\\
  &  = &  \frac{1}{2i\pi}\int_{|\lambda|=r_0} \frac{1}{2i\pi}\int_{|\mu|=r_1}(R_\lambda R_\mu f)(z)d\mu d\lambda\\
   & = & \frac{1}{2i\pi}\int_{|\lambda|=r_0} \frac{1}{2i\pi}\int_{|\mu|=r_1}\frac{ (R_\lambda f)(z)}{\mu-\lambda}d\mu d\lambda \\ 
   &  & - \frac{1}{2i\pi}\int_{|\mu|=r_1} \frac{1}{2i\pi}\int_{|\lambda|=r_0}\frac{ (R_\mu f)(z)}{\mu-\lambda} d\lambda d\mu\\
    & = &  \frac{1}{2i\pi}\int_{|\lambda|=r_0} (R_\lambda f)(z)d\lambda -0=(Pf)(z).
\end{eqnarray*}	
It follows from Lemma~\ref{lem:6.6} that $TP=PT$. Let $X_2=\ker P$, and $T_2=T_{|X_2}$. We show that 
$\sigma_W(T_2)\subset \{   \lambda\in\C :|\lambda|>r_0\}$. In fact, let $|\mu|<r_0$. Define 
\[    S_\mu:\Hol(\D)\to\Hol(\D)\mbox{ by }S_\mu f:=\frac{1}{2i\pi}\int_{|\lambda|=r_0} \frac{1}{\mu-\lambda}R_\lambda f d\lambda  . \]
Since $(\mu\Id -T)R_\lambda =(\mu-\lambda)R_\lambda +\Id$, 
\begin{eqnarray*}
(\mu\Id -T)Sf & = & \frac{1}{2i\pi}\int_{|\lambda|=r_0}R_\lambda f d\lambda + \frac{1}{2i\pi}\int_{|\lambda|=r_0}\frac{1}{\mu-\lambda} d\lambda f\\
 & = & Pf+f.
\end{eqnarray*} 
Thus, for $f\in X_2=\ker P$, $(\mu\Id-T)S_\mu f=f=S_\mu (\mu\Id -T)f$. 
This proves the claim. 

Let $X_1=P\Hol(\D)$, $T_1=T_{|X_1}$. We show that 
\[\sigma_W(T_1)\subset \{  \lambda \in\C:|\lambda |<r_0\}.\] 
In fact, let $|\mu|>r_0$. Define $\tilde{S_\mu}:\Hol(\D)\to\Hol(\D)$ by 
\[  \tilde{S_\mu} f=   \frac{1}{2i\pi}\int_{|\lambda|=r_0}\frac{1}{\lambda-\mu}R_\lambda f d\lambda .\] 
For $f \in X_1$,
\[f=Pf=\frac{1}{2i\pi}\int_{|\lambda'|=r_1} R_{\lambda'} f d\lambda'\in X_1.\]
Then 
\begin{eqnarray*}
(\mu\Id -T)	\tilde{S_\mu}f & = & (\mu\Id -T)	\tilde{S_\mu}Pf\\
& = & \frac{1}{2i\pi} \int_{|\lambda|=r_0} -R_\lambda f d\lambda + 0 \\
& = & \frac{1}{2i\pi} \int_{|\lambda|=r_0} -R_\lambda (Pf) d\lambda\\
 & = & \frac{1}{2i\pi}\int_{|\lambda|=r_0}\frac{1}{2i\pi}\int_{|\lambda'|=r_1}\frac{R_\lambda f}{\lambda'-\lambda}d\lambda' d\lambda\\
  & & - \frac{1}{2i\pi}\int_{|\lambda'|=r_1}\frac{1}{2i\pi}\int_{|\lambda|=r_0}\frac{R_\lambda' f}{\lambda'-\lambda}d\lambda d\lambda'\\
  & = &  \frac{1}{2i\pi}\int_{|\lambda|=r_0} R_\lambda fd\lambda -0=Pf=f.
\end{eqnarray*}
Since $(\mu\Id -T)\tilde{S_\mu}= 	\tilde{S_\mu}(\mu\Id -T)$, it follows that  $\tilde{S_\mu}=(\mu\Id -T)^{-1}$. We have shown 
that 
\[ \{ \mu\in\C : |\mu|>r_0\}\subset \rho_W(T_1)  \mbox{ and }  (\mu\Id -T_1)^{-1}=\tilde{S_\mu}_{|X_1}.   \]
\end{proof}	
\section{The Waelbroeck spectrum for composition operators induced by aperiodic rotations on $\Hol(\D)$}\label{sec:7}    
 Throughout this section we let $\beta\in\C$, $|\beta|=1$, such that $\beta^n\neq 1$ for all $n\in\N$. We consider $m\in\Hol(\D)$, $m\neq 0$, and define  $T\in{\mathcal L}(\Hol(D))$ by 
 \[   (Tf)(z)= m(z) f(\beta z)\mbox { for all }z\in\D .\]
 We already know from Theorem~\ref{th:2.1} that 
 \[   \sigma(T) = \beta \sigma(T)  \cup \{  m(0)\}.\]
 In particular $\beta^n m(0)\in\sigma(T)$ for all $n\in\N_0$.   
 
 Define for $0<r<1$ 
 \begin{equation}\label{eq:7.1}
 M_r:=\exp \left( \frac{1}{2\pi}\int_0^{2\pi}\log |m(re^{it})|dt\right).
 \end{equation}
 Then by Jensen's formula \cite[15.18]{Ru74}
 \begin{equation}\label{eq:7.2}
 M_r=|m_1(0)|r^N\prod_{k=1}^{m}\frac{r}{|\alpha_k|},
 \end{equation}
 where $\alpha_1,\cdots, \alpha_m$ are the zeros of $m$ in the closed unit disc $\overline{D}(0,r)\setminus \{0\}$ and where 
 $m(z)=z^N m_1(z)$ with $N\in\N_0$, $m_1\in\Hol(\D)$ such that $m_1(0)\neq 0$. In particular 
 \[  M_r=|m_1(0)|r^N\mbox{ if }m(z)\neq 0\mbox{ on }\overline{D}(0,r)\setminus \{0\}.   \] 
 It follows from (\ref{eq:7.2}) that 
 \begin{equation}\label{eq:7.3}
 M_r\leq M_s\mbox{ for }0<r<s<1.
 \end{equation}
 We let 
 \begin{equation}\label{eq:7.4}
 M_1:=\sup_{r<1} M_r \in [0,\infty].
 \end{equation}
 It follows from (\ref{eq:7.2}) that $M_1<\infty$ if and only if $\sum_k(1-|\alpha_k|)<\infty$, where now the $\alpha_k$ are the zeros of $m$ in the open unit disc $\D$ counted with multiplicities. Thus, by Weierstrass' theorem \cite[15.11]{Ru74}, there exists 
 $m\in\Hol(\D)$ such that $M_1=\infty$. However, if $m\in H^\infty(\D)$, or more generally, if $m$ is in the Nevanlinna class, then $M_1<\infty$. If $m\in H^\infty(\D)$, then 
 \begin{equation}\label{eq:7.5}
 M_1\leq M^*
 \mbox{ where }M^*:=\exp\left( \frac{1}{2\pi}\int_0^{2\pi} \log |m^*(e^{it})|dt\right) . 
 \end{equation}  
 Consider the unique (up to a constant of modulus one) factorization $m=BSF$ where $B$ is the Blaschke product associated with the zeros of $m$, $S$ is the singular inner part of $m$ and $F$ is the outer factor of $m$. Then 
 \[M_1=M^* \mbox{ if and only if }S=1,\]
 see \cite[p. 67 and p.68]{Ho62}.  At the end of this section we will give a concrete function $m$ such that in (\ref{eq:7.5}) merely the strict inequality holds. 
 
 For $0<r<1$ we let $A(r\D):={\mathcal C}(r\overline{D})\cap \Hol(r\D)$. Then 
 \[   (\tilde{T}_r  f)(z)=m(z)f(\beta z)   \]
 defines an operator $\tilde{T}_r\in {\mathcal L} (A(r\D))$. 
 \begin{prop}\label{prop:7.1}
 \begin{itemize}
 	\item[a)] $r_\sigma (\tilde{T}_r)=M_r$.
 	\item[b)] If $m(z)\neq 0$ for all $z\in r\overline{\D}$, then 
$ \sigma(\tilde{T}_r)=\{ \lambda\in\C :|\lambda|=|m(0)|  \}.  $
 	\item[c)] If $m(z_0)=0$ for some $z_0\in r\overline{\D}$, then 
	$  \sigma (\tilde{T}_r)=\{  \lambda\in\C : |\lambda|\leq M_r\}.  $
 \end{itemize}
 \end{prop} 
 \begin{proof}
 Define $\Phi:A(r\D)\to A(\D)$ by $(\Phi f)(z)=f(rz)$. Then $\Phi$ is bijective,  linear and $\Phi \tilde{T}_r \Phi^{-1}=T_r$ where 
 $T_r\in {\mathcal L}(A(\D))$ is given by 
 \[ (T_r f)(z)=m_r(z)f(\beta z),\,\forall z\in\D,  \]
 where we let $m_r(z)=m(rz)$, $(z\in\D)$. Now the claim follows from Theorem~\ref{th:5.3} and Proposition~\ref{prop:5.7}. 	
 \end{proof}	
 We first establish a spectral inclusion. Recall that the Waelbroeck spectrum $\sigma_W(T)$ is closed and $\sigma(T)\subset \sigma_W(T)$. 
 \begin{prop}\label{prop:7.2}  
 \begin{itemize}
 	\item[a)] If $|\lambda|>M_r$ for all $r<1$, then $\lambda\in\rho(T)$. 
 	\item[b)] $\sigma(T)\subset \{  \lambda\in\C : |\lambda|<M_1\}$ if $m$ has a zero in $\D$.
 	\item[c)] $\sigma_W(T)\subset \{  \lambda\in\C :|\lambda|\leq M_1\}$ if $M_1<\infty$.  
 \end{itemize}
 \end{prop}
 \begin{proof}
 a) Let $|\lambda|>M_r$ for all $0<r<1$. Let $g\in\Hol(\D)$. Given $0<r<1$, by Proposition~\ref{eq:7.1} a), there exists a unique $f_r\in A(r\D)$ satisfying 
 \[  \lambda f_r(z)-f_r(\beta z)=g(z) \mbox{ for all }z\in r\D .\]
 Thus, for $r<r'<1$, ${f_{r'}}_{|r\D}=f_r$. Defining $f(z):=f_r(z)$ if $|z|\leq r$ we obtain a function $f\in\Hol(\D)$ such that 
 \[  \lambda f(z)-m(z)f(\beta z) =g(z) \]
 for all $z\in\D$; i.e. $\lambda f-Tf=g$. Since $f_{|r\overline{\D}} \in A(r\D)$, uniqueness follows from the fact that $\lambda\Id -\tilde{T}_r$ is injective.  This proves a). \\
 b) If $m$ has a zero in $\D$, then, by (\ref{eq:7.2}), $M_r<M_1$ for all $r<1$. So the claim follows from a).\\
 c) Let $|\lambda_0|>M_1$. Let $\delta>0$ such that $\overline{D}(\lambda_0,\delta)\subset \{\lambda\in\C:  |\lambda|>M_1\}$. 
 Let $0<r<1$, $f\in\Hol(\D)$. We have to show that 
 \[     \sup_{|\lambda-\lambda_0|\leq \delta}\sup_{|z|\leq r}  |(R_\lambda f)(z)|<\infty .   \]   
 But 
 \[(R_\lambda f)(z)=((\lambda \Id -\tilde{T}_r )^{-1} f_{|A(r\D)})(z)\mbox{ for }|z|\leq r.\]  	
 Since $r_\sigma (\tilde{T}_r)\leq M_r<M_1$ the claim follows. 
 \end{proof}	
\begin{thm}\label{th:7.3}
Assume that $m(z)\neq 0$ for all $z\in\D$. Then 
\begin{itemize}
	\item[a)] $\{\beta^n m(0):n\in\N_0\}\subset \sigma(T)\subset \{  \lambda\in\C :|\lambda|=|m(0)| \}$.
	\item[b)] $\sigma_W(T)=\{ \lambda\in\C :|\lambda|=|m(0)| \}$.  
\end{itemize}
\end{thm}
 \begin{proof}
 The first inclusion is Corollary~\ref{cor:2.2}. Since $m(z)\neq 0$ for all $z\in\D$, $M_r=|m(0)|$ for all	$0<r<1$, see (\ref{eq:7.2}). Hence $M_1=|m(0)|$ and Proposition~\ref{prop:7.2} c) shows that $\sigma_W(T)\subset \{ \lambda\in\C :|\lambda|\leq |m(0)| \}$. But $T$ is invertible and 
 \[(T^{-1}f)(z)=\frac{1}{m(\overline{\beta} z)}f(\overline{\beta} z).\]  
 From the identity $(\mu\Id -T^{-1})=\frac{T}{\mu}(T-\frac{1}{\mu})^{-1}$ one sees that $\sigma_W(T^{-1})=\sigma_W (T)^{-1}$. It follows from the first part of the proof applied to $T^{-1}$ that 
 $\frac{1}{|\mu|}\leq \frac{1}{|m(0)|}$ for all $\mu\in \sigma_W(T)$. Hence $|\mu|\geq |m(0)|\geq |\mu|$ for all $\mu\in\sigma_W(T)$. This proves that 
 \[   \sigma_W(T) \subset \{   \lambda\in\C : |\lambda|=|m(0)|  \}, \]
 and the first part of a) implies the other inclusion since $\sigma_W(T)$ is closed. Thus b) is proved. The second inclusion in a) follows from b) since $\sigma(T)\subset \sigma_W(T)$.   
 \end{proof}
Next we want to determine the Waelbroeck spectrum of $T$ when $m$ has zeros in $\D$. 

We first prove rotational invariance of the Waelbroeck spectrum. 
\begin{lem}\label{lem:7.4}
\begin{itemize}
\item[a)] Let $\lambda\in\rho_W(T)$. Then $\beta\lambda\in\rho_W(T)$. 
\item[b)] If $\lambda\in\sigma_W(T)$, then $\T \lambda\subset\sigma_W(T)$. 
\item[c)] If $\lambda\in\rho_W(T)$, then $\T \lambda \subset \rho_W(T)$.  
\end{itemize}	
\end{lem} 	
\begin{proof}
1. By Corollary~\ref{cor:2.2} a) $\lambda\in\sigma(T)$ implies that 
\[  \{\beta^n\lambda:n\in\N\}\subset \sigma(T)\subset \sigma_W(T). \]
Since $\sigma_W(T)$ is closed, it follows that $\T \lambda \subset \sigma_W(T)$. \\
2. We prove a). To that aim, let $\lambda\in\rho_W(T)$, and we show that $\beta\lambda\in \rho_W(T)$.   \\
There exists $\delta>0$ such that   $D(\lambda, \delta)\subset \rho(T)$ and for all $f\in\Hol(\D)$, 
\begin{equation}\label{eq:7.6}
\sup_{|\lambda-\mu|< \delta}\sup_{|z|\leq r}|R_\mu f(z)|<\infty.
\end{equation}	
By 1. this implies that $\beta\mu\in\rho(T)$   for all $\mu\in D(\lambda,\delta )$. Note that $\beta D(\lambda,\delta)=D(\beta \lambda,\delta)$.  Thus it suffices to show that for all $r<1$, 
\begin{equation}\label{eq:7.7}
\sup_{|\lambda-\mu|< \delta}\sup_{|z|\leq r}|R_{\beta\mu} g(z)|<\infty.
\end{equation}
 for all $g\in\Hol(\D)$. Let $g\in \Hol(\D)$, $0<r<1$. \\
 \emph{First case}: $g(0)=0$. Then there exists $g_1\in\Hol(\D)$ such that $g=e_1g_1$. By (\ref{eq:5.5}) we have
 \begin{eqnarray*}
 \sup_{|\lambda-\mu|< \delta}\sup_{|z|\leq r} |R_{\beta \mu}g(z)| & =  & 	\sup_{|\lambda-\mu|< \delta}\sup_{|z|\leq r}  |e_1 R_\mu (\overline{\beta}g_1))(z)|\\
  & \leq & \sup_{|\lambda-\mu|< \delta} r|(R_\mu g_1)(z)|<\infty
\end{eqnarray*}	
 by (\ref{eq:7.6}).  \\
 \emph{Second case}: $g=e_0$. By (\ref{eq:5.6}),
 \[ \sup_{|\lambda-\mu|<\delta}\sup_{|z|\leq r} \left\vert (R_{\beta \mu} e_0) (z) \right\vert   =\]
 
 \[ \sup_{|\lambda-\mu|<\delta}\sup_{|z|\leq r} \left|   \frac{1}{\beta \mu -m(0)} + (R_{\beta\mu} (m-m(0))f(0))(z)\right|<\infty  \]

 by the first case. Since $\Hol(\D)=\Hol_0(\D)\oplus \C e_0$, (\ref{eq:7.7}) is proved.\\
 3. We prove b). Let $\lambda\in \sigma_W(T)$. Then by a) $\overline{\beta}\lambda \in \sigma_W(T)$. Hence 
 \[  \{  \overline{\beta}^k\lambda:k\in\N \}\subset \sigma_W(T).  \]
 This implies that $\lambda\T\subset \sigma_W(T)$. \\
 4. c) follows from b). 
 \end{proof}	
\begin{thm}\label{th:7.5}
Let $\beta\in\C$, $|\beta|=1$, $\beta^n\neq 1$ for all $n\in\N$. Let $m\in \Hol(\D)$ such that $m(z_0)=0$ for some $z_0\in\D$. Consider $T:\Hol(\D)\to \Hol(\D)$ given by 
\[  (Tf)(z)=m(z)f (\beta z)\,\,\, z\in\D . \]
Then $\sigma_W(T)=\overline{D}(0,M_1)$ if $M_1<\infty$ and $\sigma_W(T)=\C$ if $M_1=\infty$. 
\end{thm}
\begin{proof}
a) We already know from Proposition~\ref{prop:7.2}	 that $\sigma_W(T)\subset \overline{D}(0,M_1)$ if $M_1<\infty$.\\
b) Assume that there exists $\lambda_0\in\C$ such that $0<|\lambda_0|<M_1$, $\lambda_0\not\in \sigma_W(T)$. Then by Lemma~\ref{lem:7.4}, $\lambda_0\T\subset \rho_W(T)$.  By the spectral decomposition theorem, Theorem~\ref{th:6.3}, there exist closed subspaces $X_1,X_2$ of $\Hol(\D)$ such that $TX_j\subset X_j$ and, for $T_j=T_{|X_j}$, $j=1,2$, one has
\[   \sigma_W (T_1)   = \{    \lambda \in\sigma_W (T) :|\lambda|<|\lambda_0|  \} \mbox{ and }\sigma_W (T_2)   = \{    \lambda \in\sigma_W (T) :|\lambda|>|\lambda_0|  \} .    \]    
Assume that $X_2=\{ 0\}$. Then $X_1=X$ and $|\lambda|<|\lambda_0|$ for all $\lambda\in\sigma_W(T)$. Let $0<r<1$ such that 
$|\lambda_0|<M_r$ and such that $m(z)\neq 0$ whenever $|z|=r$. Consider the operator $\tilde{T}_r\in{\mathcal L}(A(r\D))$ from Proposition~\ref{prop:7.1}. Then $r_\sigma (\tilde{T}_r)=M_r$. For $|\lambda|>M_r$, 
\begin{equation}\label{eq:7.8}
g_\lambda (z):=((\lambda\Id -\tilde{T}_r)^{-1}e_0)(z)=\sum_{n\geq 0}\frac{m_n(z)}{\lambda^{n+1}}
\end{equation} 
defines a function $g_\lambda\in A(r\D)$. Moreover, for $|z|=r$, $\lim_{n\to\infty} |m_n(z)|^{1/n}=M_r$ by Lemma~\ref{lem:5.5}. Thus, for fixed $z\in\C$ with $|z|=r$, the series  (\ref{eq:7.8}) has a singular point $\lambda_1\in\C$ with $|\lambda_1|=M_r$. This means that the function $\lambda\mapsto g_\lambda (z)$ does not have a holomorphic extension to an open set $\Omega$ containing $\{ \lambda\in\C :|\lambda|>M_1  \}\cup \{ \lambda_1\}$. \\
We show that $\lambda_1\in \sigma_W(T)$. In fact, assume that $\lambda_1\in \rho_W(T)$. Then $\lambda\mapsto (R_\lambda e_0)(z)$ is holomorphic on $\rho_W(T)$ and coincides with $g_\lambda$ on $\{\lambda\in \C :|\lambda|>M_1\}$. This is a contradiction. So we have proved that $\lambda_1\in\sigma_W(T)$. This implies that $X_2\neq \{ 0\}$. Let $0\neq f\in X_2$. Since $T_2$ is bijective for all $n\in\N$, there exists $g_n\in X_2$ such that $f=T_2^n g_2$. Since $m$ has a zero $z_0$ in $\D$, this implies   $f=0$ as in the proof of Theorem~\ref{th:5.3}, a contradiction.      
	
\end{proof}

We want to give a concrete example where $M_1<M^*$ (see the beginning of Section~\ref{sec:7}). 
\begin{exam}\label{ex:7.6}
	Let 
	\begin{equation}\label{eq:7.9}
	m(z)=(1-z)\exp\left (   -\frac{1+z}{1-z}\right).
	\end{equation}
	Then $m\in A(\D)$. One has $M_1=|m(0)|=\frac{1}{e}$ since $m(z)\neq 0$ for all $z\in\D$. We claim that $M_1<M^*$. 
	Indeed, $|m(e^{it})|=|1-e^{it}|=2|\sin(t/2)|$, and thus 
	\begin{eqnarray*}
	 M^* & = & \exp \left(  \frac{1}{2\pi} \int_0^{2\pi} \log |2(\sin(t/2))|dt  \right)\\
		    & = & 2\exp \left(     \frac{1}{2\pi} \int_0^{\pi}\log |2(\sin(u))|2du\ \right)\\
		    & = & 2\exp \left(     \frac{2}{\pi} \int_0^{\pi/2}\log (\sin(\theta))d\theta \right)\\
		    & \geq & 2\exp \left(     \frac{2}{\pi} \int_0^{\pi/2}\log (2\theta /\pi))d\theta \right)\\
		    & = & 2\exp \left(     \frac{2}{\pi} \int_0^{1}\log (u)\frac{\pi}{2}du \right)\\
		    & = & \frac{2}{e},
\end{eqnarray*}	  
where we used the fact that $\sin (\theta)\geq \frac{2}{\pi} \theta$ on $[0,\pi/2]$ by concavity. Thus $M^*\geq \frac{2}{e}>\frac{1}{e}=M_1$. 
Consequently, for the function $m$ given by (\ref{eq:7.4}), 
\[  \sigma(T_{A(\D)})=D(0,M^*)\mbox{ whereas }\sigma(T)=D(0,1/e)   \]
is a smaller disc.  
If $M_1<|\lambda|\leq M^*$, then $(\lambda\Id -T)$ is invertible on $\Hol(\D)$ but 
\[   (\lambda\Id -T)^{-1}A(\D)\not\subset A(\D).  \]
For this function $m$ we also have 
\[  r_\sigma (T_{H^p(\D)})  =M^*.  \]
Thus 
\[   (\lambda\Id -T)^{-1} H^p(\D)\not\subset H^p(\D) ,\,\,\, 1\leq p\leq \infty,   \]
either when $M_1< \vert \lambda \vert \leq M^*$. 
\end{exam}
We do not know whether in Theorem~\ref{th:7.5}
\[   \sigma_W(T)=\overline{\sigma(T)}.\]
This is the case if $m$ is constant. If $m$ is not constant, it can happen that the spectrum itself coincides with the open unit disc of radius $M_1$. We give an example.
  \begin{exam}\label{ex:7.7}
  	Let $\beta\in\C$, $|\beta|=1$, $\beta^n \neq 1$ for all $n\in\N$, $m(z)=z$. 
  	Then $\sigma(T)=\{  \lambda\in\C : |\lambda|<1\}$. Note that $M_r=r$ for $r \in ]0;1[$. Thus $M_1=1$. 
  \end{exam} 
\begin{proof}
    Since $0 = m(0)$, we have $0 \in \sigma(T)$ (see Theorem \ref{th:2.1}).
    
    Now, let $\lambda \neq 0$ with $\vert \lambda \vert < 1 $. We claim that $\lambda \in \sigma(T)$. Indeed, assume that $\lambda \notin \sigma (T)$. Then, there exists $f \in \Hol(\D)$ such that
    \begin{equation}\label{eq:7.42}
        z f(\beta z) - \lambda f(z) = e_0(z) ~ (\forall z \in \D).
    \end{equation}
    Let's write $f(z) = \sum_{n \geqslant 0} a_n z^n$. Then, equation (\ref{eq:7.42}) becomes
    \[ \sum_{n \geqslant 0} \beta^n a_n z^{n+1} - \lambda \sum_{n \geqslant 0} a_n z_n = 1, \]
    which gives
    \[
    \left\{  
    \begin{array}{l}
    - \lambda a_0 = b_0 \\
    \forall n \in \N_0, ~ \beta^n a_n - \lambda a_{n+1} = 0
     \end{array}.
    \right.
    \]
    We now get $a_n = \frac{- \beta^{\frac{n(n-1)}{2}}}{\lambda^{n+1}}$ for all $n \in \N_0$. Hence, $\limsup_{n \rightarrow + \infty} \vert a_n \vert^{\frac{1}{n}} = \frac{1}{\vert \lambda \vert} >1$, thus providing a contradiction.
    
    We now have proved that $\left\{ \lambda \in \C ~:~ \vert \lambda \vert < 1 \right\} \subset \sigma(T)$. Furthermore, since $M_r = r$ for $0<r<1$ and $M_1 = 1$, we have by Proposition \ref{prop:7.2} (b) that $\sigma(T) \subset \left\{ \lambda \in \C ~:~ \vert \lambda \vert < 1 \right\}$.
\end{proof}	

\section{Weighted composition operators induced by elliptic automorphisms}\label{sec:8} 
In this concluding section we describe some of our results if the rotation is replaced by an arbitrary elliptic automorphism with a unique fixed point $\alpha\in\D$. This means that  we consider $\varphi:\D\to\D$ given by 
\[  \varphi(z)=\Psi_\alpha \circ r_\beta \circ \Psi_\alpha\]
where $\beta\in\T$, $\beta\neq 1$, $r_\beta(z)=\beta z$ and $\Psi_\alpha (z)=\frac{\alpha-z}{1-\overline{\alpha}z}=\Psi_\alpha^{-1}(z).$
Thus $\varphi$ is conjugated to the rotation $r_\beta$. 

Let $m\in\Hol(\D)$ and define 
\[ T:\Hol(\D)\to\Hol(\D)\mbox{ by }Tf=mf\circ\varphi.   \]
Let $U:\Hol(\D)\to\Hol(\D)$ be given by 
$  Uf=f\circ \Psi_\alpha.$
Then $U$ is an isomorphism and $U=U^{-1}$. Let 
$\tilde{T}=UTU^{-1}=UTU.$
Then 
\[  \tilde{T}f=\tilde{m} f\circ r_\beta \mbox{ where }\tilde{m}=m\circ \Psi_\alpha.\]
Thus $T$ is similar to a weighted composition operator induced by a rotation $\tilde{T}$ as considered before and consequently
\[  \sigma(T)=\sigma(\tilde{T}),\,\,\, \sigma_W(T)=\sigma_W(\tilde{T}).\]
We call $\varphi$ \emph{aperiodic} if $\varphi_n\neq \Id$ for all $n\in\N$, where $\varphi_n:=\varphi\circ \cdots \circ \varphi$ ($n$ times). This is equivalent to $\beta^n\neq 1$ for all $n\in\N$. 
Thus Theorem~\ref{th:7.3} and \ref{th:7.5} give the following result.
\begin{thm}\label{th:8.1}
Assume that $\varphi$ is aperiodic.
\begin{itemize}
\item[a)] If $m(z)\neq 0$ for all $z\in\D$, then 
\[\sigma_W(T)=\{\lambda\in\C:|\lambda|=|m(\alpha)|  \}.\]
\item[b)] If there exists $z_0\in\D$ such that $m(z_0)=0$, then there exists $0<R\leq \infty$ such that 
\[  \sigma_W(T)=\{ \lambda\in\C:|\lambda|\leq R\}\mbox{ if }R<\infty 
\mbox{ and }  \sigma_W(T)=\C\mbox{ otherwise}. \]
\end{itemize}
\end{thm}
Concerning the point spectrum we merely note Proposition~\ref{prop:3.3} and \ref{prop:3.6} in the situation considered here.
\begin{thm}\label{th:8.2}
If $m(z_0)=0$ for some $z_0\in \D$, then $\sigma_p(T)=\emptyset$. This is always valid, no matter whether $\varphi$ is periodic or aperiodic.
\end{thm}
 
\noindent \textbf{Acknowledgments:}  This research is partly supported by the B\'ezout Labex, funded by ANR, reference ANR-10-LABX-58.    
We are grateful to Professor Vasilescu for information on the spectral projection in Fr\'echet spaces.  

\bibliographystyle{abbrv}

\end{document}